\documentclass[12pt,a4, reqno]{amsart}
\usepackage[a4paper, left=26mm, right=26mm, top=28mm, bottom=34mm]{geometry}
\usepackage[all]{xy}
\usepackage{amsmath}
\usepackage{amscd}
\usepackage{amssymb}
\usepackage{amsthm}
\usepackage{url}
\usepackage{mathrsfs}
\usepackage{comment}
\usepackage{bbm}

\newtheorem{thm}{Theorem}[section]
\newtheorem{lem}[thm]{Lemma}
\newtheorem{cor}[thm]{Corollary}
\newtheorem{prop}[thm]{Proposition}
\newtheorem{conj}[thm]{Conjecture}

\theoremstyle{definition}

\newtheorem{rem}[thm]{Remark}
\newtheorem{defn}[thm]{Definition}

\newtheorem{ex}[thm]{Example}

\def\F{{\mathbb F}}
\def\G{{\mathbb G}}
\def\Q{{\mathbb Q}}
\def\R{{\mathbb R}}
\def\Z{{\mathbb Z}}
\def\C{{\mathbb C}}

\def\Br{\mathop{\mathrm{Br}}\nolimits}
\def\Fr{\mathop{\mathrm{Fr}}\nolimits}
\def\End{\mathop{\mathrm{End}}\nolimits}

\def\Aut{\mathop{\mathrm{Aut}}\nolimits}

\def\Cl{\mathop{\mathrm{Cl}}\nolimits}

\def\Frob{\mathop{\mathrm{Frob}}\nolimits}

\def\Gal{\mathop{\mathrm{Gal}}\nolimits}

\def\Hom{\mathop{\mathrm{Hom}}\nolimits}

\def\id{\mathop{\mathrm{id}}\nolimits}

\def\GL{\mathop{\mathrm{GL}}\nolimits}

\def\GSpin{\mathop{\mathrm{GSpin}}\nolimits}
\def\SO{\mathop{\mathrm{SO}}\nolimits}

\def\Pic{\mathop{\mathrm{Pic}}\nolimits}

\def\Spec{\mathop{\rm Spec}}

\def\Tr{\mathop{\text{\rm Tr}}\nolimits}

\def\cris{\text{\rm cris}}

\def\rank{\mathop{\text{\rm rank}}\nolimits}

\def\det{\mathop{\mathrm{det}}\nolimits}

\def\KS{\mathop{\mathrm{KS}}\nolimits}

\def\p{\mathop{\mathfrak{p}}\nolimits}

\newcommand{\et}{\mathrm{\acute{e}t}}

\def\num{\mathop{\mathrm{num}}\nolimits}
\def\prim{\mathop{\mathrm{prim}}\nolimits}

\def\CH{\mathop{\mathrm{CH}}\nolimits}

\def\alg{\mathop{\mathrm{alg}}\nolimits}
\def\tr{\mathop{\mathrm{tr}}\nolimits}

\numberwithin{equation}{section}

\usepackage[T1]{fontenc}
\makeatletter
\@namedef{subjclassname@2020}{\textup{2020} Mathematics Subject Classification}
\makeatother

\begin{document}

\title[Hodge standard conjecture for self-products of K3 surfaces]{The Hodge standard conjecture for self-products of K3 surfaces}

\author{Kazuhiro Ito}
\address{Kavli Institute for the Physics and Mathematics of the Universe (WPI), The University of Tokyo,
5-1-5 Kashiwanoha, Kashiwa, Chiba, 277-8583, Japan}
\email{kazuhiro.ito@ipmu.jp}

\author{Tetsushi Ito}
\address{Department of Mathematics, Faculty of Science, Kyoto University, Kyoto 606-8502, Japan}
\email{tetsushi@math.kyoto-u.ac.jp}

\author{Teruhisa Koshikawa}
\address{Research Institute for Mathematical Sciences, Kyoto University, Kyoto 606-8502, Japan}
\email{teruhisa@kurims.kyoto-u.ac.jp}




\maketitle

\begin{abstract}
As an application of our previous work on CM liftings of K3 surfaces and the Tate conjecture, we prove the Hodge standard conjecture for squares of K3 surfaces. We also deduce the Hodge standard conjecture for all the powers of certain K3 surfaces.  
\end{abstract}

\section{Introduction} \label{Section:Introduction}
The standard conjectures on algebraic cycles is a set of conjectures envisioned by Grothendieck \cite{Grothendieck, Kleiman68, Kleiman94}. The Hodge standard conjecture---one of the standard conjectures---is modelled on the classical Hodge theory and the Hodge index theorem for surfaces, and it predicts certain positivity of the intersection product with respect to a fixed polarization. 
Grothendieck wrote at the end of \cite{Grothendieck} that
\begin{quote}
Alongside the problem of resolution of singularities, the proof of the standard conjectures seems to me to be the most urgent task in algebraic geometry. 
\end{quote}
For instance, it was observed that the Hodge standard conjecture for $X\times X$, together with the Lefschetz standard conjecture for $X$, implies the Weil conjecture for $X$. In fact, the standard conjectures provide a good theory of pure motives and it has been quite influential. While the Weil conjecture was proved by Deligne by other means and the semisimplicity of the category of numerical motives was directly shown by Jannsen \cite{Jannsen92}, the Hodge standard conjecture itself has seen very little progress. 

Let us recall the precise statement.
Let $k$ be a field and $X$ a geometrically connected smooth projective variety of dimension $d$ over $k$.
We fix a prime number $\ell$ invertible on $X$.
Let $A^i_\ell(X)$ denote the $\Q$-vector subspace of
$H^{2i}_\et(X_{\overline{k}}, \Q_\ell(i))$
spanned by the image of the cycle class map
\[
Z^i(X) \to H^{2i}_\et(X_{\overline{k}}, \Q_\ell(i)),
\]
where $\overline{k}$ is an algebraic closure of $k$, $X_{\overline{k}}:=X \times_{\Spec k} \Spec \overline{k}$, and $Z^i(X)$ is the abelian group of algebraic cycles of codimension $i$ on $X$.
Let $\mathscr{L}$ be an ample line bundle on $X$.
For an integer $i \leq d/2$, we define the primitive algebraic part as
\[
A^{i, \prim}_\ell(X) := \{ \, \alpha \in A^i_\ell(X) \, \vert \, \alpha \cdot \mathscr{L}^{d-2i+1} = 0 \, \},
\]
where $\alpha \cdot \mathscr{L}^{d-2i+1} \in A^{d-i+1}_\ell(X)$ denotes the cup product.

\begin{conj}[The Hodge standard conjecture]\label{Conjecture:Hodge standard}
The pairing
\[
\langle -, - \rangle_i \colon A^{i, \prim}_\ell(X) \times A^{i, \prim}_\ell(X) \to \Q, \quad (\alpha, \beta) \mapsto (-1)^i \alpha \cdot \beta \cdot  \mathscr{L}^{d-2i}
\]
is \emph{positive definite}.
\end{conj}

The conjecture in characteristic $0$ follows from the Hodge theory and it holds true for surfaces by the Hodge index theorem. Some cases of abelian varieties over finite fields can be reduced to the Hodge index theorem or the Hodge theory as pointed out by Milne \cite{Milne02, Milne22}; see also Remark \ref{Remark:Milne} below. 
The second author studied the behavior of the conjecture under blow-ups during the study of the weight-monodromy conjecture for certain $p$-adically uniformized varieties \cite{Ito}. 
To the best of our knowledge, there are no other results for the Hodge standard conjecture itself for an arbitrary given $\ell$. 

There are some other related results, however. Yun and Zhang obtained a result for certain moduli spaces of shtukas that seems to be related to the Hodge standard conjecture \cite[Theorem 1.7]{Yun-Zhang}. 
The recent work of Ancona \cite{Ancona} and a complementary article by the third author \cite{Koshikawa} study the \emph{numerical} version for certain abelian varieties. 
(The numerical version for an abelian variety over a finite field implies the $\ell$-adic homological version for infinitely many $\ell$ by \cite{Clozel}, but the Tate conjecture would be essentially necessary to treat all $\ell$.)

In this paper, we establish the following result as an application of our previous work \cite{Ito-Ito-Koshikawa} on CM liftings of K3 surfaces and the Tate conjecture. 

\begin{thm}\label{Theorem:Hodge standard for squares, intro}
Let $X$ be a K3 surface over a field.
The Hodge standard conjecture (Conjecture \ref{Conjecture:Hodge standard}) holds true for $X^2:=X \times X$ and every ample line bundle on $X^2$.
\end{thm}

\begin{rem}\label{Remark:Conjecture B}
It is well known that the Lefschetz standard conjecture, more precisely, Conjecture B in \cite{Grothendieck, Kleiman68, Kleiman94}, holds true for $X$, in fact, for any products of surfaces; see \cite[Corollary 2A10, Corollary 2.5]{Kleiman68}. 
Let us recall again that if $X$ is defined over a finite field, the Hodge standard conjecture for $X^2$ and the Lefschetz standard conjecture for $X$ imply the Weil conjecture for $X$, and Theorem \ref{Theorem:Hodge standard for squares, intro} justifies the assumption. However, we did not attempt to avoid the Weil conjecture in our proof of Theorem \ref{Theorem:Hodge standard for squares, intro}. 
\end{rem}

\begin{cor}\label{Corollary:standard conjecture for squares}
Let $X$ be a K3 surface over a field $k$.
All the standard conjectures hold for $X^2$. In particular, numerical equivalence coincides with $\ell$-adic homological equivalence for algebraic cycles on $X^2$. 
Namely, the kernel of the cycle class map
\[
Z^i(X^2) \to H^{2i}_\et(X^2_{\overline{k}}, \Q_\ell(i))
\]
is the subgroup of algebraic cycles that are numerically equivalent to $0$. 
Moreover, 
\[
A^i_{\ell}(X^2)\otimes_{\Q}\Q_{\ell} \to H^{2i}_\et(X^2_{\overline{k}}, \Q_\ell(i))
\]
is injective. 
\end{cor}

\begin{proof}
The Hodge standard conjecture (Theorem \ref{Theorem:Hodge standard for squares, intro}) and Conjecture B in Remark \ref{Remark:Conjecture B} imply all other standard conjectures \cite[Theorem 4-1, Proposition 5-1]{Kleiman94}. 
The last part follows from \cite[Lemma 2.5]{Tate94}. 
\end{proof}

The Hodge standard conjecture for squares also implies the following:

\begin{cor}[Conservativity]
Let $X, Y$ be K3 surfaces over a field $k$, and suppose a correspondence $u \in Z^2 (X\times Y)$ induces a Galois-equivariant isomorphism
\[
u_{\ell}\colon 
H^{2}_\et(X_{\overline{k}}, \Q_\ell) \xrightarrow{\cong}
H^{2}_\et(Y_{\overline{k}}, \Q_\ell). 
\]
Then, $u$ induces an \emph{isomorphism} between the $\ell$-adic homological motives of the second cohomology of $X$ and $Y$.
\end{cor}

Note that K\"unneth projectors exist for $X$ and $Y$, and hence the $\ell$-adic homological motives of the second cohomology of $X$ and $Y$ are well-defined.

\begin{proof}
It suffices to show that $u$ (or rather its restriction to the $\ell$-adic homological motives of the second cohomology) admits both left and right inverses in the category of $\ell$-adic homological motives. This follows from Theorem \ref{Theorem:Hodge standard for squares, intro} and \cite[Corollary 3.14]{Kleiman68}.
\end{proof}

\begin{rem}
If $k$ is finite, then, for any prime number $\ell'\neq p$, $\ell'$-adic homological equivalence coincides with numerical equivalence for algebraic cycles on $X\times Y$. Indeed, the Tate conjecture for $X^2$ \cite{Ito-Ito-Koshikawa} and the existence of $u$ imply the Tate conjecture for $X\times Y$. Together with the semisimplicity of Frobenius, this gives the claim. Therefore, $u$ induces an isomorphism of numerical motives of the second cohomology of $X$ and $Y$, and, in turn, induces an isomorphism between $\ell'$-adic second cohomology of $X$ and $Y$. 
\end{rem}

Let us sketch our proof of Theorem \ref{Theorem:Hodge standard for squares, intro}. It is easily reduced to the case where $k$ is a finite field, and we may assume so. In this setup, recall that we proved the Tate conjecture for $X^2$ in \cite{Ito-Ito-Koshikawa}. The Tate conjecture and the semisimplicity of Frobenius imply that numerical equivalence coincides with $\ell$-adic homological equivalence for every $\ell$; see Corollary \ref{Corollary:Conjecture D for squares over finite fields} (a special case of Theorem \ref{Corollary:standard conjecture for squares}). This enables us to study the motive of $X$. In particular, we prove in Theorem \ref{Theorem:structure of endomorphism ring} that the \emph{transcendental} part of the numerical motive of $X$ is simple in the category of numerical motives. This has been known for us for years. A key observation toward Theorem \ref{Theorem:Hodge standard for squares, intro} is that a certain formula for the intersection product on $X^2$ involving a natural involution $X^2 \to X^2 ; (x_0, x_1)\mapsto (x_1, x_0)$, which is also used in \cite{Jannsen92}, reduces the Hodge standard conjecture to the positivity of the involution restricted to the transcendental part. 
We deduce the positivity of the involution from the positivity of the Rosati involution of the Kuga--Satake abelian variety associated with $X$. To do so, we shall relate two involutions and this is done by using the main results in \cite{Ito-Ito-Koshikawa}: our proof of the Tate conjecture for $X^2$ in \cite{Ito-Ito-Koshikawa} actually yields a surjection of algebraic groups over $\Q$
\[
I \to \textnormal{``motivic isometries of the transcendental part''}, 
\]
where $I$ is Kisin's algebraic group as in \cite[Section 8.1]{Ito-Ito-Koshikawa} (denoted by $I^{\KS}$ in the main body of this paper) inside the automorphism group of the Kuga--Satake abelian variety attached to $X$.
As each side is related to the corresponding involution respectively, this surjection enables us to translate the positivity of the Rosati involution. 

Note that our lifting result in \cite{Ito-Ito-Koshikawa} translated using the result of the previous paragraph provides a lifting of a given cycle on $X^2$ to characteristic $0$ if the $\Q$-subalgebra generated by the cycle inside the motivic endomorphism ring of $X$ is stable under the involution; cf.\ Theorem \ref{Theorem:CM lifting revisited}. This in particular shows that the Hodge standard conjecture is true for such cycles on $X^2$ by a different way. However, as far as we can see, some additional trick like the one above is necessary to complete the proof of Theorem \ref{Theorem:Hodge standard for squares, intro}. 

\begin{rem}
It would be clear that this work is closely related to Ziquan Yang's note \cite{Yang}, which provides a variant of the argument of \cite{Ito-Ito-Koshikawa} (in characteristic $p \geq 5$). 
Let us only mention the following consequence. 
Let $X$ be a K3 surface over a finite field of characteristic $p\geq5$. 
According to \cite[Remark 3.9]{Yang}, the standard conjectures for $X^2$ established in this work imply that ``derived transcendental $\ell$-isogenies\footnote{Derived isogenies in his sense are compositions of isogenies that are induced from equivalences of derived categories of certain coherent sheaves.}'' span the endomorphism ring of the transcendental motive of $X$; this is an analogue of Huybrechts's result for the Hodge conjecture \cite{HuybrechtsMotive}. 
(Our proof of the Tate conjecture for $X^2$ \cite{Ito-Ito-Koshikawa} with \cite{HuybrechtsMotive} only tells us that specializations of ``derived isogenies'' in characteristic 0 span the endomorphism ring.)  
\end{rem}

\begin{rem}\label{Remark:Milne}
One may compare our work with results of Milne \cite{Milne99, Milne02, Milne22}. 
In \cite{Milne99, Milne02}, it is shown that the Hodge conjecture for all abelian varieties with complex multiplication implies the Tate and Hodge standard conjectures for abelian varieties over finite fields. In more recent \cite{Milne22}, he explained that the Hodge conjecture for the powers of a single abelian variety with complex multiplication satisfying a certain condition implies the Tate and Hodge standard conjectures for the powers of its reduction, giving unconditional interesting results as an application of the work of Markman on Weil classes \cite{Markman}. On the other hand, our works may be summarized that the Hodge conjecture for the squares of all K3 surfaces with complex multiplication (proved by Mukai and Buskin \cite{Buskin}) implies the Tate and Hodge standard conjectures for the squares of K3 surfaces over finite fields. 
\end{rem}

Finally, let us consider higher powers $X^n$ of $X$. 
In some situations, it is possible to reduce the Hodge standard conjecture for $X^n, n\geq 3$ to the case of $X^2$. 
For instance, we obtain the following result. 

\begin{thm}[Corollary \ref{Corollary:Hodge standard for Picard number 17}]
Let $X$ be a K3 surface over an algebraically closed field with Picard number $\rho (X)\geq 17$. 
Conjecture \ref{Conjecture:Hodge standard} holds true for an arbitrary power $X^n$ and every ample line bundle on $X^n$. 
\end{thm}

A more general version (over finite fields) is formulated in Theorem \ref{Theorem:Hodge standard conjecture for neat K3 surface} with the notion of \emph{neat} K3 surfaces introduced in Definition \ref{Definition:neat K3}. 
The notion of neat K3 surfaces essentially appeared in previous works such as \cite{Milne19, Zarhin93}, and in particular Zarhin proved the Tate conjecture for arbitrary powers of an ordinary K3 surface over a finite field by showing that any ordinary K3 surface is neat. 
(See Subsection \ref{Subsection:Examples} for further examples of neat (and nonneat) K3 surfaces.)
We basically apply the same idea to the Hodge standard conjecture. 

Let us note that the corresponding notion is studied more in the context of abelian varieties over finite fields, e.g. \cite{Zarhin94, Zarhin15}. The condition $\rho (X)\geq 17$ above is inspired by the case of abelian surfaces considered in \cite{Zarhin94}. 

This paper is organized as follows.
In Section \ref{Section:The transcendental motive of a K3 surface over a finite field}, we define the (numerical) transcendental motive of a K3 surface over a finite field and investigate some basic properties of it.
In Section \ref{Section:CM liftings of K3 surfaces revisited},
we show the positivity of the natural involution of the endomorphism ring of the transcendental motive.
Along the way, we prove a slightly refined version of our result on CM liftings obtained in \cite{Ito-Ito-Koshikawa}.
We also give a reformulation of it without involving the Kuga--Satake construction.
In Section \ref{Section:The Hodge standard conjecture for the square of a K3 surface}, we complete the proof of Theorem \ref{Theorem:Hodge standard for squares, intro}.
Finally, in Section \ref{Section:Self-products of neat K3 surfaces},
we introduce and study neat K3 surfaces over finite fields, and prove the Hodge standard conjecture for arbitrary powers of neat K3 surfaces.

\section{The transcendental motive of a K3 surface over a finite field}\label{Section:The transcendental motive of a K3 surface over a finite field}

We study the transcendental motive of a K3 surface $X$ over a finite field in this section.
We especially prove that if $X$ is of finite height, then the numerical transcendental motive $\mathfrak{t}(X)$ of $X$ is simple, and the endomorphism ring $\End(\mathfrak{t}(X))$ is a finite dimensional division algebra over $\Q$ whose center is the $\Q$-subalgebra generated by the Frobenius; see Theorem \ref{Theorem:structure of endomorphism ring}.
In Section \ref{Subsection:Remarks on the transcendental motives in characteristic 0}, some related remarks on the transcendental motive of a K3 surface in characteristic $0$ are discussed.

\subsection{The transcendental motive}\label{Subsection:The transcendental motive}

Let $\F_q$ be a finite field with $q$ elements of characteristic $p>0$ and $\overline{\F}_q$ an algebraic closure of $\F_q$.
Let $X$ be a K3 surface over $\F_q$,
$X^2:=X \times_{\Spec \F_q} X$, and $X^2_{\overline{\F}_q}:=X^2 \times_{\Spec \F_q} \Spec \overline{\F}_q$.
(Throughout this paper,
for a scheme $X$ over a commutative ring $R$ and a ring homomorphism $R \to S$,
let $X_S$ denote the base change $X \times_{\Spec R} \Spec S$.
We use the same notation for the base change of an algebraic group, a module, etc.
We sometimes omit the subscript when no confusion can arise.)

We begin by recalling the following theorem from \cite{Ito-Ito-Koshikawa}.
The abelian group of algebraic cycles of codimension $i$ on $X^2$ 
is denoted by $Z^i(X^2)$.

\begin{thm}[\cite{Ito-Ito-Koshikawa}]\label{Theorem:Tate conjecture for square}
The Tate conjecture holds true for $X^2$.
Namely, the $\ell$-adic cycle class map
\[
Z^i(X^2)\otimes_\Z \Q_\ell \to H^{2i}(X^2, \Q_\ell)(i)^{\Gal(\overline{\F}_q/\F_q)}
\]
is surjective for every $i$ and every $\ell \neq p$,
where we write $H^{2i}(X^2, \Q_\ell):=H^{2i}_\et(X^2_{\overline{\F}_q}, \Q_\ell)$ and $(i)$ denotes the Tate twist.
The same statement holds for the crystalline cycle class maps.
\end{thm}

\begin{proof}
See \cite[Theorem 1.5]{Ito-Ito-Koshikawa}.
\end{proof}

\begin{cor}\label{Corollary:Conjecture D for squares over finite fields}
Numerical equivalence coincides with $\ell$-adic homological equivalence for algebraic cycles on $X^2$ for every $\ell \neq p$.
The same statement holds for the crystalline cohomology.
\end{cor}

\begin{proof}
This is essentially in \cite{Ito-Ito-Koshikawa} (see \cite[Lemma 10.4]{Ito-Ito-Koshikawa} in particular), and follows from Theorem \ref{Theorem:Tate conjecture for square} and the semisimplicity of the action of Frobenius on the $\ell$-adic cohomology of
$X^2$; see \cite[Theorem 2.9]{Tate94}.
\end{proof}

Let $A^i(X^2)$ denote the $\Q$-vector space of $\Q$-linear algebraic cycles of codimension $i$ on $X^2$ modulo numerical equivalence.
We obtain from Corollary \ref{Corollary:Conjecture D for squares over finite fields} an inclusion
\[
A^i(X^2) \hookrightarrow H^{2i}(X^2, \Q_\ell)(i).
\]
We say that a class in $H^{2i}(X^2, \Q_\ell)(i)$ is algebraic if it lies in the image of the above injection, in which case we regard it as an element in $A^i(X^2)$.

Let
$\mathscr{M}_{\num}(k)$
denote the category of numerical motives over a field $k$ (with coefficients in $\Q$).
We now recall the definition of the transcendental motive
$\mathfrak{t}(X) \in \mathscr{M}_{\num}(\F_q)$
of $X$ (compare with the discussion in \cite[Proposition 7.2.3]{Kahn-Murre-Pedrini}).
The argument below also applies to a K3 surface over the algebraic closure $\overline{\F}_q$ of $\F_q$.
First, we have the K\"unneth decomposition of the diagonal
\[
\Delta_X = \pi_0 + \pi_2 +\pi_4
\]
in $A^2(X^2)$, where $\pi_m$ is an idempotent inducing the projection to $H^{m}(X, \Q_\ell)$ and independent of $\ell \neq p$ for every $m$.

The $\ell$-adic Chern class map gives an embedding
$\Pic(X_{\overline{\F}_q})_{\Q_\ell} \hookrightarrow H^2(X, \Q_\ell)(1)$.
Let $T(X)_{\ell}$ denote the orthogonal complement of $\Pic(X_{\overline{\F}_q})_{\Q_\ell}$
with respect to the cup product pairing on $H^2(X, \Q_\ell)(1)$.
Since the intersection pairing
$\langle -, - \rangle$
is perfect on
$\Pic(X_{\overline{\F}_q})_{\Q}$,
we have the following orthogonal decomposition
\begin{equation}\label{equation:orthogonal decomposition}
\Pic(X_{\overline{\F}_q})_{\Q_\ell} \oplus T(X)_{\ell}=H^2(X, \Q_\ell)(1).
\end{equation}
We choose an orthogonal basis $\{ D_i \}_i$  of $\Pic(X_{\overline{\F}_q})_{\Q}$, and put
\[
\pi_{\mathrm{alg}}:= \sum_{i} \frac{1}{\langle D_i, D_i \rangle} p^*_1D_i \cdot p^*_2D_i \in
A^2(X^2_{\overline{\F}_q}),
\]
where $p_1, p_2 \colon X^2 \to X$ denote the projections (and their base changes).
This correspondence is an idempotent satisfying $\pi_{\mathrm{alg}}=\pi_{\mathrm{alg}}\pi_2=\pi_2 \pi_{\mathrm{alg}}$ and induces the projection
$
H^2(X, \Q_\ell)(1) \to \Pic(X_{\overline{\F}_q})_{\Q_\ell}
$
given by $(\ref{equation:orthogonal decomposition})$.
Since $\pi_{\mathrm{alg}}$ is invariant under the action of Frobenius,
it can be regarded as an element in $A^2(X^2)$.
Then we put
$
\pi_{\tr}:=\pi_2-\pi_{\alg}
$
and the \textit{transcendental motive} of $X$
is defined to be
\[
\mathfrak{t}(X):=(X, \pi_{\tr}, 1) \in \mathscr{M}_{\num}(\F_q).
\]
The $\ell$-adic realization of $\mathfrak{t}(X)$ is well-defined by Corollary \ref{Corollary:Conjecture D for squares over finite fields}. In fact, by its construction, the $\ell$-adic realization is $T(X)_{\ell}$.

Let $W(\F_q)$ denote the ring of Witt vectors of $\F_q$.
The discussion above also applies to the crystalline cohomology and the crystalline realization of $\mathfrak{t}(X)$ is the transcendental part $T(X)_{\cris}$ of $H^2_{\cris}(X)(1)$, where $H^2_{\cris}(X):=H^2_\cris(X/W(\F_q))[1/p]$ and $(1)$ denotes the Tate twist.

\subsection{The height of a K3 surface}\label{Subsection:The height of a K3 surface}

Let $\widehat{\Br}(X)$ denote the \textit{formal Brauer group} of $X$, which is a one-dimensional smooth formal group over $\F_q$; see \cite{Artin-Mazur} (and also \cite[Section 6.3]{Ito-Ito-Koshikawa}).
The \textit{height} $h$ of $X$ is by definition the height of the formal group $\widehat{\Br}(X)$.
The height satisfies $1 \leq h \leq 10$ or $h=\infty$,
and we say that $X$ is of finite height if $1 \leq h \leq 10$ and $X$ is supersingular if $h=\infty$.

\begin{rem}\label{Remark:supersingular case}
The transcendental motive $\mathfrak{t}(X)$ is zero if and only if the rank of $\Pic(X_{\overline{\F}_q})$ is $22$.
Recall that the Tate conjecture for $X$ holds true; this is proved in \cite{MadapusiPera} for $p \geq 3$ and in \cite{Kim-MadapusiPera, Ito-Ito-Koshikawa, KimMadapusiPera-Erratum} for all $p \geq 2$ (see also
\cite[Introduction]{MadapusiPera} for previously known results).
As a result,
the rank of $\Pic(X_{\overline{\F}_q})$ is $22$ if and only if $X$ is supersingular.
(The same result holds for K3 surfaces over any algebraically closed field of characteristic $p>0$; see \cite[Theorem 4.8, Proposition 6.17]{Liedtke} and also \cite[Lemma 10.6]{Ito-Ito-Koshikawa} for more details.)
\end{rem}

We fix an embedding $\overline{\Q} \subset \overline{\Q}_p$.
Let
$\nu_q \colon \overline{\Q}_p \to \Q \cup \{ \infty \}$
denote the $p$-adic valuation normalized by $\nu_q(q)=1$.

\begin{prop}\label{Proposition:L-function of K3}
Assume that $X$ is of finite height.
Then the polynomial
\[
L_{\mathrm{trc}}(X, T):=\det(1-\Frob_q T \vert T(X)_{\ell})
\]
satisfies the following properties:
\begin{enumerate}
    \item $L_{\mathrm{trc}}(X, T)$ is in $\Q[T]$ and $\Z_\ell[T]$ for every $\ell \neq p$, and independent of $\ell \neq p$, and all its complex roots have absolute value one.
    No root of $L_{\mathrm{trc}}(X, T)$ is a root of unity.
    \item $L_{\mathrm{trc}}(X, T)=Q(T)^e$ for an irreducible polynomial $Q(T) \in \Q[T]$ of degree $2d$ and a positive integer $e >0$.
    \item There exists a positive integer $h >0$ such that $\nu_q(\alpha) \in \{ 1/h, 0, -1/h \}$ for every root $\alpha \in \overline{\Q} \subset \overline{\Q}_p$ of $L_{\mathrm{trc}}(X, T)$.
    Moreover, there are exactly $h$ roots $\alpha$ with $\nu_q(\alpha)=-1/h$.
    \item We put
    \[
    Q_{<0}(T):=\prod_{\nu_q(\alpha)=-1/h}(1-\alpha T),
    \]
    where $\alpha \in \overline{\Q} \subset \overline{\Q}_p$ runs over the roots of $Q(T)$ satisfying $\nu_q(\alpha)=-1/h$.
    Then $Q_{<0}(T)$ is an irreducible polynomial in $\Q_p[T]$.
\end{enumerate}
\end{prop}

\begin{proof}
(1) follows from the Tate and Weil conjectures for $X$ and the semisimplicity of Frobenius.
The remaining statements are checked in the proof of \cite[Proposition 3.2]{Yu-Yui}.
See also \cite[Theorem 1]{Taelman}.
\end{proof}

\begin{defn}\label{Definition:d, e, h}
Let $X$ be a K3 surface of finite height over $\F_q$.
We write
\[
d(X):=d, \quad e(X):=e, \quad \text{and} \quad  h(X):=h
\]
for the positive integers $d, e, h$ in Proposition \ref{Proposition:L-function of K3}.
\end{defn}

\begin{rem}
For a finite extension $\F_{q^m}$ of $\F_q$, $h(X_{\F_{q^m}})=h(X)$ holds, while $d(X_{\F_{q^m}}) \leq d(X)$, $e(X) \leq e(X_{\F_{q^m}})$ and
the inequalities can be strict; see Example \ref{Example:d and e change}.
\end{rem}

Let $K_0:=W(\F_q)[1/p]$.
Recall that for an $F$-isocrystal $M$ over $K_0$,
we have a slope decomposition
\[
M = \bigoplus_{\lambda \in \Q} M_\lambda
\]
where $M_{\lambda}$ is an $F$-isocrystal over $K_0$ which has a single slope $\lambda \in \Q$.
Assume that $X$ is of finite height.
Then, for the $F$-isocrystal $T=T(X)_{\cris}$ over $K_0$,
we have
\[
T(X)_{\cris} = T_{1/h(X)} \oplus T_0 \oplus T_{-1/h(X)}
\]
and we can identify $T_{-1/h(X)}$ with $\mathrm{TC}(\widehat{\Br}_X)_\Q(1)$; see for instance \cite[Proposition 7]{Taelman}.
Here $\mathrm{TC}(\widehat{\Br}_X)$ is the Cartier-Dieudonn\'e module of typical curves of $\widehat{\Br}_X$.
The height of $X$ is therefore $h(X)$.

\subsection{The endomorphism ring of the transcendental motive}\label{Subsection:The endomorphism ring of the transcendental motive}

We know that the category
$\mathscr{M}_{\num}(\F_q)$
is a semisimple abelian category by \cite{Jannsen92}, and hence the endomorphism ring
$\End(\mathfrak{t}(X))$ of $\mathfrak{t}(X)$
is a finite-dimensional semisimple $\Q$-algebra.
The purpose of this subsection is to establish some basic properties of $\End(\mathfrak{t}(X))$.

First, let us note the following fact:

\begin{prop}\label{Proposition:endomorphism ring realization}
Let
$\End_{\Frob_q}(T(X)_\ell)$
denote the $\Q_\ell$-vector space of endomorphisms of $T(X)_\ell$ that commute with the geometric Frobenius $\Frob_q$.
The $\ell$-adic realization, for every $\ell \neq p$, induces an isomorphism 
\begin{equation}\label{equation:endomorphism ring realization}
\End(\mathfrak{t}(X)) \otimes_\Q \Q_\ell \overset{\sim}{\to} \End_{\Frob_q}(T(X)_\ell).
\end{equation}
\end{prop}

\begin{proof}
This follows from Theorem \ref{Theorem:Tate conjecture for square} and Corollary \ref{Corollary:Conjecture D for squares over finite fields}.
\end{proof}

Let $\Fr_X \in \End(\mathfrak{t}(X))$ denote the endomorphism induced by the Frobenius endomorphism
$X \to X$ over $\F_q$.

\begin{thm}\label{Theorem:structure of endomorphism ring}
Let $X$ be a K3 surface of finite height over $\F_q$.
Then the transcendental motive $\mathfrak{t}(X)$ is a simple object in $\mathscr{M}_{\num}(\F_q)$.
Moreover the $\Q$-subalgebra $\Q[\Fr_X] \subset \End(\mathfrak{t}(X))$ generated by $\Fr_X$ is a field and the endomorphism ring $\End(\mathfrak{t}(X))$ is a central division algebra over $\Q[\Fr_X]$.
\end{thm}

\begin{proof}
The isomorphism $(\ref{equation:endomorphism ring realization})$, together with the semisimplicity of Frobenius and the bicommutant theorem, implies that $\Q[\Fr_X]$ is the center of $\End(\mathfrak{t}(X))$.

It remains to prove that $\mathfrak{t}(X)$ is a simple object.
Assume that $\mathfrak{t}(X)$ is not a simple object.
Then, since $\mathscr{M}_{\num}(\F_q)$ is semisimple, we have $\mathfrak{t}(X)=V \oplus W$ for some nonzero objects $V$ and $W$ in $\mathscr{M}_{\num}(\F_q)$. 
By Corollary \ref{Corollary:Conjecture D for squares over finite fields}, the crystalline realizations $V_\cris$, $W_\cris$ of $V$, $W$ are well-defined and we obtain the decomposition
$T(X)_{\cris}=V_{\cris} \oplus W_{\cris}$. 
The characteristic polynomials of $\Frob_q$ acting on $V_\cris$ and $W_\cris$ both lie in $\Q[T]$; see \cite[Theorem 2]{Katz-Messing}.
Thus they are of the form $Q(T)^m$ and $Q(T)^n$ for the irreducible polynomial $Q(T) \in \Q[T]$ appearing in Proposition \ref{Proposition:L-function of K3}.
Here $m, n>0$ are positive integers.
It follows that
the negative slope parts
$V_{\cris, -1/h}, W_{\cris, -1/h}$ of $V_\cris, W_\cris$
are nonzero, where $h$ is the height of $X$.
Therefore, we obtain a nontrivial decomposition
\[
\mathrm{TC}(\widehat{\Br}_X)_\Q(1)=T(X)_{\cris, -1/h}=V_{\cris, -1/h} \oplus W_{\cris, -1/h}
\]
and this gives a contradiction since
$\mathrm{TC}(\widehat{\Br}_X)_\Q$ is a simple object in the category of $F$-isocrystals over $K_0$.
\end{proof}

\begin{cor}\label{Corollary:rank of motivic endomorphism ring}
For $X$ of finite height,
$[\Q[\Fr_X]: \Q]=2d(X)$ and
$\dim_{\Q[\Fr_X]} \End(\mathfrak{t}(X))=e(X)^2$ hold.
(See Definition \ref{Definition:d, e, h} for $d(X)$ and $e(X)$.)
\end{cor}

\begin{proof}
This follows from Theorem \ref{Theorem:structure of endomorphism ring} and the isomorphism $(\ref{equation:endomorphism ring realization})$.
\end{proof}

\subsection{Remarks on transcendental motives in characteristic $0$}\label{Subsection:Remarks on the transcendental motives in characteristic 0}

Let $K$ be a field of characteristic $0$ and $\mathscr{M}_{\hom}(K)$ the category of homological motives over $K$ (with coefficients in $\Q$) with respect to the $\ell$-adic cohomology theory for some prime number $\ell$.
Since $K$ is of characteristic $0$,
it is known that this category is independent of $\ell$, or in other words, homological equivalence does not depend on the choice of $\ell$. Indeed, we may assume that $K=\C$ by the Lefschetz principle, and then Artin's comparison theorem between Betti and $\ell$-adic cohomology for varieties over $\C$ implies the independence of $\ell$.

By repeating the same construction as in Section \ref{Subsection:The transcendental motive}, we define the homological transcendental motive $\mathfrak{t}(Y)$ in $\mathscr{M}_{\hom}(K)$ for a K3 surface $Y$ over $K$.

\begin{lem}\label{Lemma:algebraically closed fields extension}
For an extension of algebraically closed fields $L/K$ of characteristic $0$ and a K3 surface $Y$ over $K$, we have
    \[
    \End_{\mathscr{M}_{\hom}(K)}(\mathfrak{t}(Y)) \overset{\sim}{\to} \End_{\mathscr{M}_{\hom}(L)}(\mathfrak{t}(Y_{L})).
    \]
\end{lem}

\begin{proof}
The injectivity is clear.
The surjectivity follows from usual spreading out and specialization arguments.
In fact, the result holds for any homological motive in $\mathscr{M}_{\hom}(K)$ and its base change to $L$ (with the same proof).
\end{proof}

Let $Y$ be a K3 surface over a field $K$ of characteristic $0$.
The numerical motive attached to the homological transcendental motive will be denoted by the same notation, so we write $\mathfrak{t}(Y) \in  \mathscr{M}_{\num}(K)$.

\begin{prop}\label{Proposition:transcendental motive is simple in characteristic 0}
\begin{enumerate}
    \item
    $\End_{\mathscr{M}_{\hom}(K)}(\mathfrak{t}(Y)) \overset{\sim}{\to} \End_{\mathscr{M}_{\num}(K)}(\mathfrak{t}(Y)).
    $
    \item $\mathfrak{t}(Y)$ is a simple object in $\mathscr{M}_{\num}(K)$ and $\End_{\mathscr{M}_{\num}(K)}(\mathfrak{t}(Y))$ is a number field, in particular, it is commutative.
\end{enumerate}
\end{prop}

\begin{proof}
(1) Let us remark that numerical equivalence coincides with homological equivalence for algebraic cycles on $Y \times Y$ as Corollary \ref{Corollary:standard conjecture for squares} is already known in characteristic 0 (with the same argument).
The assertion follows from this fact.

(2) We may assume that $K$ is algebraically closed.
Thanks to Lemma \ref{Lemma:algebraically closed fields extension},
we can further reduce to the case where $K=\C$.
In this case, it is known that the orthogonal complement $T(Y)_\Q$
of $\Pic(Y)_\Q \subset H^2(Y, \Q(1))$
is an irreducible $\Q$-Hodge structure and, moreover, the ring
$\End_{\mathrm{Hdg}}(T(Y)_\Q)$
of endomorphisms of $T(Y)_\Q$ that preserve the $\Q$-Hodge structure on it is a number field; see \cite[Theorem 1.6]{Zarhin83}.
This fact implies (2).
\end{proof}

\begin{rem}\label{Remark:compare with characteristic p case}
    In contrast, for a K3 surface $X$ over $\F_q$ of finite height,
    the ring $\End(\mathfrak{t}(X))$ is not commutative in general.
    In fact, it is commutative if and only if $e(X)=1$.
\end{rem}

\section{CM liftings of K3 surfaces revisited}\label{Section:CM liftings of K3 surfaces revisited}

In this section, we discuss some consequences of \cite{Ito-Ito-Koshikawa}.
One of the main results is Corollary \ref{Corollary:positive involution}, which states that the involution of $\End(\mathfrak{t}(X))$ obtained by switching the two factors of $X \times X$ is positive.
This is a key point in our proof of the Hodge standard conjecture for the square of a K3 surface (Theorem \ref{Theorem:Hodge standard for squares, intro}).
We also rewrite our result on CM liftings of K3 surfaces given in \cite{Ito-Ito-Koshikawa} without involving the Kuga--Satake construction; see Theorem \ref{Theorem:CM lifting revisited}.
Its formulation is conceptually much simpler than the previous one.

\subsection{Motivic isometries}\label{Subsection:Motivic isometries}

Let $X$ be a K3 surface over $\F_q$.
The transpose
$A^2(X^2) \to A^2(X^2)$, $\alpha \mapsto {^t}\alpha$
obtained by interchanging the two factors of $X^2$
induces an involution
\[
\iota \colon \End(\mathfrak{t}(X)) \to \End(\mathfrak{t}(X)), \quad f \mapsto {^t}f.
\]
Under the isomorphism $(\ref{equation:endomorphism ring realization})$,
the involution $\iota_{\Q_\ell}$ coincides with the one of $\End_{\Frob_q}(T(X)_\ell)$ defined as taking the adjoint with respect to the cup product pairing on $T(X)_\ell$.
Let
$I(\mathfrak{t}(X))$
denote the algebraic group over $\Q$ whose $R$-valued points are
\[
I(\mathfrak{t}(X))(R) = \{ \, f \in (\End(\mathfrak{t}(X))\otimes_\Q R)^\times \, \vert \, f\iota(f)=1 \, \}
\]
for every commutative ring $R$ over $\Q$.

Let $\mathscr{L}$ be an ample line bundle on $X$.
We also consider the primitive part
\[
\p(X):=\mathfrak{h}^{2, \prim}(X)(1):=(X, \pi_2-\pi_\mathscr{L}, 1) \in \mathscr{M}_{\num}(\F_q)
\]
of the motive $\mathfrak{h}^2(X)(1)$, where
\[
\pi_\mathscr{L}:=\langle \mathscr{L}, \mathscr{L} \rangle^{-1} p^*_1 \mathscr{L} \cdot p^*_2 \mathscr{L} \in A^2(X^2).
\]
The $\ell$-adic realization of $\p(X)$ is the orthogonal complement
$P^2(X, \Q_\ell)(1)$ of
(the first Chern class of) $\mathscr{L}$ in $H^2(X, \Q_\ell)(1)$.
As in the case of $\End(\mathfrak{t}(X))$,
we can define an involution
\[
\iota \colon \End(\p(X)) \to \End(\p(X)), \quad f \mapsto {{^t}f},
\]
again denoted by $\iota$,
and then let
$\mathrm{O}(\p(X))$ denote
the algebraic group over $\Q$ defined by
\[
\mathrm{O}(\p(X))(R) = \{ \, f \in (\End(\p(X))\otimes_\Q R)^\times \, \vert \, f\iota(f)=1 \, \}
\]
for every commutative ring $R$ over $\Q$.

\begin{lem}\label{Lemma:orthogonal group realization}
\begin{enumerate}
    \item The $\ell$-adic realization induces an isomorphism of algebraic groups over $\Q_\ell$
\begin{equation}\label{equation:transcendental orthogonal group realizarion}
    I(\mathfrak{t}(X))_{\Q_\ell} \overset{\sim}{\to} \SO_{\Frob_q}(T(X)_\ell),
\end{equation}
where
$\SO_{\Frob_q}(T(X)_\ell)$  
is the subgroup of the special orthogonal group
$\SO(T(X)_\ell)$
consisting of the elements which commute with $\Frob_q$.
Moreover $I(\mathfrak{t}(X))$ is connected.
\item The $\ell$-adic realization induces isomorphism of algebraic groups over $\Q_\ell$
\begin{equation}\label{equation:primitive orthogonal group realizarion}
\mathrm{O}(\p(X))_{\Q_\ell} \overset{\sim}{\to} \mathrm{O}_{\Frob_q}(P^2(X, \Q_\ell)(1)),
\end{equation}
where $\mathrm{O}_{\Frob_q}(P^2(X, \Q_\ell)(1))$ is the centralizer of $\Frob_q$ in the orthogonal group $\mathrm{O}(P^2(X, \Q_\ell)(1))$.
\end{enumerate}
\end{lem}

\begin{proof}
(1) By $(\ref{equation:endomorphism ring realization})$, we obtain
$I(\mathfrak{t}(X))_{\Q_\ell} \overset{\sim}{\to} \mathrm{O}_{\Frob_q}(T(X)_\ell)$.
We show that
$\mathrm{O}_{\Frob_q}(T(X)_\ell)$
is connected; this implies (1).
Since no eigenvalue of $\Frob_q$ acting on $T(X)_\ell$ is a root of unity, we may write $\alpha_1, \alpha^{-1}_1, \dotsc, \alpha_d, \alpha^{-1}_d$ for the distinct eigenvalues.
Let $T_i \subset T(X)_\ell \otimes_{\Q_\ell} \overline{\Q}_\ell$ denote the eigenspace corresponding to $\alpha_i$.
Then we have the following description
\[
\mathrm{O}_{\Frob_q}(T(X)_\ell)_{\overline{\Q}_\ell} \cong \prod^d_{i=1} \GL(T_i)
\]
since the action of $\Frob_q$ on $T(X)_\ell$ is semisimple; see also the proof of \cite[Lemma 10.10]{Ito-Ito-Koshikawa}.
In particular $\mathrm{O}_{\Frob_q}(T(X)_\ell)$ is connected.

(2)
The $\ell$-adic realization induces an isomorphism
\[
\End(\p(X)) \otimes_\Q \Q_\ell \overset{\sim}{\to} \End_{\Frob_q}(P^2(X, \Q_\ell)(1))
\]
similar to $(\ref{equation:endomorphism ring realization})$,
and $(\ref{equation:primitive orthogonal group realizarion})$ follows.
\end{proof}

Let $(\mathscr{L})^\perp$ denote the orthogonal complement of $\mathscr{L}$ in $\Pic(X_{\overline{\F}_q})_\Q$, on which $\Frob_q$ acts.
We have a natural decomposition
\[
\mathrm{O}(\p(X)) = \mathrm{O}_{\Frob_q}((\mathscr{L})^\perp) \times I(\mathfrak{t}(X)),
\]
where $\mathrm{O}_{\Frob_q}((\mathscr{L})^\perp)$ is the centralizer of
$\Frob_q$ in the orthogonal group
$\mathrm{O}((\mathscr{L})^\perp)$
of the quadratic space $(\mathscr{L})^\perp$ over $\Q$.
We define
\[
I(\p(X)):= \SO_{\Frob_q}((\mathscr{L})^\perp) \times I(\mathfrak{t}(X)),
\]
so we have by $(\ref{equation:transcendental orthogonal group realizarion})$
\begin{equation}\label{equation:neutral component realization}
    I(\p(X))_{\Q_\ell} \cong \SO_{\Frob_q}(P^2(X, \Q_\ell)(1)).
\end{equation}
If the action of $\Frob_q$ on $(\mathscr{L})^\perp$ is trivial, then $I(\p(X))$ is the neutral component of $\mathrm{O}(\p(X))$.

\subsection{Relative homological motives}\label{Subsection:Relative homological motives}

Let $k$ be a finite field $\F_q$ or $\overline{\F}_q$,
and let $K$ be a finite totally ramified extension of $W(k)[1/p]$ with ring of integers $\mathcal{O}_K$.
We will describe our results using
the category
$\mathscr{M}_{\hom}(\mathcal{O}_K)$
of relative homological motives over $\mathcal{O}_K$.
The objects of
$\mathscr{M}_{\hom}(\mathcal{O}_K)$
are the triples $(\mathcal{X}, q, m)$, where $\mathcal{X}$ is a smooth projective scheme over $\mathcal{O}_K$ of constant relative dimension, say $d$,
\[
q \in \CH^d(\mathcal{X} \times_{\Spec \mathcal{O}_K} \mathcal{X})_\Q / \sim_{\hom}
\]
is an idempotent, and $m$ is an integer, and the morphisms are given by
\[
\Hom_{\mathscr{M}_{\hom}(\mathcal{O}_K)}((\mathcal{X}, q, m), (\mathcal{Y}, r, n))=r(\CH^{d+n-m}(\mathcal{X} \times_{\Spec \mathcal{O}_K} \mathcal{Y})_\Q / \sim_{\hom}) q.
\]
Here $\sim_{\hom}$ denotes the homological equivalence with respect to the $\ell$-adic cohomology theory
$\mathcal{X} \mapsto H^*_\et(\mathcal{X}_{\overline{K}}, \Q_\ell)$ for some prime number $\ell$, where $\mathcal{X}_{\overline{K}}=\mathcal{X} \times_{\Spec \mathcal{O}_K} \Spec \overline{K}$ is the geometric generic fiber.
This coincides with the homological equivalence with respect to the Betti cohomology theory
$\mathcal{X} \mapsto H^*(\mathcal{X}_{\C}, \Q)$
for any embedding $\mathcal{O}_K \hookrightarrow \C$.
We have the following functors
\[
\mathscr{M}_{\hom}(\mathcal{O}_K) \to \mathscr{M}_{\hom}(k) \to \mathscr{M}_{\mathrm{num}}(k),
\]
where $\mathscr{M}_{\hom}(k)$ is the category of $\ell$-adic homological motives over $k$.
Let
\[
\mathrm{sp} \colon \mathscr{M}_{\hom}(\mathcal{O}_K)\to \mathscr{M}_{\mathrm{num}}(k)
\]
denote the composition of the above two functors.

There is a relative version of the (homological) transcendental motive.
Let $\mathcal{X}$ be a smooth projective scheme over $\mathcal{O}_K$ whose fibers are K3 surfaces.
In this paper, we call such a scheme $\mathcal{X}$ a projective K3 surface over $\mathcal{O}_K$.
With the same arguments as in Section \ref{Subsection:The transcendental motive},
we can define the homological transcendental motive
\[
\mathfrak{t}(\mathcal{X}_K)=(\mathcal{X}_K, \pi_{\tr}, 1) \in \mathscr{M}_{\hom}(K)
\]
for the generic fiber $\mathcal{X}_K$; see also Section \ref{Subsection:Remarks on the transcendental motives in characteristic 0}.
By using
\begin{equation}\label{equation:relative homological Chow group}
    \CH^{2}(\mathcal{X} \times_{\Spec \mathcal{O}_K} \mathcal{X})_\Q / \sim_{\hom} \overset{\sim}{\to} \CH^{2}(\mathcal{X}_K \times_{\Spec K} \mathcal{X}_K)_\Q / \sim_{\hom},
\end{equation}
we can regard the idempotent $\pi_{\tr}$ as an idempotent in $\CH^{2}(\mathcal{X} \times_{\Spec \mathcal{O}_K} \mathcal{X})_\Q/ \sim_{\hom}$.
Then we put
\[
\mathfrak{t}(\mathcal{X}):=(\mathcal{X}, \pi_{\tr}, 1) \in \mathscr{M}_{\hom}(\mathcal{O}_K).
\]

Let $\mathcal{L}$ be a relatively ample line bundle on $\mathcal{X}$.
Let $\p(\mathcal{X}_K) \in \mathscr{M}_{\hom}(K)$ denote the primitive part of the homological motive
$\mathfrak{h}^2(\mathcal{X}_K)(1)$ with respect to $\mathcal{L}_K$.
Similarly,
we can define
the relative version
$\p(\mathcal{X}) \in \mathscr{M}_{\hom}(\mathcal{O}_K)$
of $\p(\mathcal{X}_K)$.
We have
$\mathrm{sp}(\p(\mathcal{X}))=\p(\mathcal{X}_k)$ in $\mathscr{M}_{\mathrm{num}}(k)$, where $\p(\mathcal{X}_k)$ is the primitive part of $\mathfrak{h}^2(\mathcal{X}_k)(1)$ with respect to $\mathcal{L}_k$.
Thus we obtain a specialization map
$
\End(\p(\mathcal{X})) \to \End(\p(\mathcal{X}_k)).
$

On the other hand,
the specialization $\mathrm{sp}(\mathfrak{t}(\mathcal{X}))$
need not be isomorphic to the numerical motive $\mathfrak{t}(\mathcal{X}_k)$.
In fact, we have
$\mathrm{sp}(\mathfrak{t}(\mathcal{X}))=\mathfrak{t}(\mathcal{X}_k)$
if and only if
the specialization map
$
\Pic(\mathcal{X}_{\overline{K}})_\Q \to \Pic(\mathcal{X}_{\overline{\F}_q})_\Q
$
is bijiective.
In this case, we have a specialization map
\begin{equation}\label{equation:specialization of motivic homomorphisms}
    \End(\mathfrak{t}(\mathcal{X})) \to \End(\mathfrak{t}(\mathcal{X}_{k})).
\end{equation}
(In fact, one can define a natural map $\End(\mathfrak{t}(\mathcal{X})) \to \End(\mathfrak{t}(\mathcal{X}_{k}))$ even if $
\Pic(\mathcal{X}_{\overline{K}})_\Q \to \Pic(\mathcal{X}_{\overline{\F}_q})_\Q
$ is not bijective.)

\begin{rem}\label{Remark:endomorphism ring relative case}
It follows from $(\ref{equation:relative homological Chow group})$ that
\begin{equation}\label{equation:endomorphism ring relative case}
\End(\mathfrak{t}(\mathcal{X})) \overset{\sim}{\to} \End(\mathfrak{t}(\mathcal{X}_K)) \quad \text{and} \quad \End(\p(\mathcal{X})) \overset{\sim}{\to} \End(\p(\mathcal{X}_K)).
\end{equation}
In particular, we see that $\End(\mathfrak{t}(\mathcal{X}))$ is a number field by Proposition \ref{Proposition:transcendental motive is simple in characteristic 0}.
\end{rem}

\subsection{The Kuga--Satake construction}\label{Subsection:The Kuga--Satake construction}

Let us briefly recall a few facts about the Kuga--Satake construction from \cite{MadapusiPera, Kim-MadapusiPera, Ito-Ito-Koshikawa}.
We follow the setting of \cite{Ito-Ito-Koshikawa}.
Let $X$ be a K3 surface over $\F_q$.
As in \cite[Section 6.2]{Ito-Ito-Koshikawa},
after replacing $\F_q$ by its finite extension, we can
\begin{itemize}
    \item find a primitive ample line bundle $\mathscr{L}$ on $X$ (in the sense of \cite[Section 3]{MadapusiPera}), and
    \item attach a certain abelian variety $A$ over $\F_q$, called the Kuga--Satake abelian variety,
    to the polarized K3 surface $(X, \mathscr{L})$.
\end{itemize}
We have the following natural embedding compatible with Frobenius actions
\begin{equation}\label{equation:Kuga--Satake embedding}
    P^2(X, \Q_\ell)(1) \hookrightarrow \End_{\Q_\ell}(T_\ell(A)_\Q),
\end{equation}
where
$T_\ell(A)_\Q$ is the rational Tate module of $A_{{\overline{\F}}_q}$.

Following Kisin \cite{KisinModp},
we can attach a certain reductive group $I^{\KS}$ over $\Q$ to $(X, \mathscr{L})$; see \cite[Section 8.1]{Ito-Ito-Koshikawa} (it was denoted by $I$ there).
It is a subgroup of $(\End_{\overline{\F}_q}(A_{\overline{\F}_q})\otimes_\Z \Q)^\times$ and contains the subgroup of scalars $\G_m$.
Moreover, it acts on $P^2(X, \Q_\ell)(1)$ via the embedding $(\ref{equation:Kuga--Satake embedding})$ and we obtain a homomorphism over $\Q_\ell$
\begin{equation}\label{equation:I to SO(P)}
   I^{\KS}_{\Q_\ell} \to \SO(P^2(X, \Q_\ell)(1)).
\end{equation}
The image of this map is contained in the centralizer
of $\Frob_{q^m}$ in
$\SO(P^2(X, \Q_\ell)(1))$
for a sufficiently divisible $m$.

\begin{rem}\label{Remark:GSpin and SO}
We put $P_\ell:=P^2(X, \Q_\ell)(1)$.
The algebraic group $I^{\KS}_{\Q_\ell}$ can be (noncanonically) embedded into the general spin group $\GSpin(P_\ell)$ such that
$(\ref{equation:I to SO(P)})$ factors through the usual homomorphism
\[
\GSpin(P_\ell) \to \SO(P_\ell)
\]
induced by the adjoint action of $\GSpin(P_\ell)$ on the Clifford algebra $\Cl(P_\ell)$.
Let us remark that,
following Madapusi Pera,
we construct the Kuga--Satake abelian variety $A$ in a slightly different setting from usual in order to treat the case where the degree of the polarization $\mathscr{L}$ is divisible by $p$.
In particular, $\dim_{\Q_\ell} T_\ell(A)_\Q=2^{22}$ is larger than $\dim_{\Q_\ell}\Cl(P_\ell)=2^{21}$.
\end{rem}

The following theorem is a direct consequence of the main results in \cite{Ito-Ito-Koshikawa}.

\begin{thm}\label{Theorem:compare with Kisin's algebraic group}
Assume that $X$ is of finite height.
We let
\[
\tau_\ell \colon I^{\KS}_{\Q_\ell} \to \SO_{\Frob_{q^m}}(P^2(X, \Q_\ell)(1))
\]
denote the map induced from $(\ref{equation:I to SO(P)})$ for a sufficiently divisible $m$.
Then the following assertions hold.
\begin{enumerate}
    \item $\tau_\ell$ is surjective. In particular, we have
    \[
    \SO_{\Frob_{q^m}}(P^2(X, \Q_\ell)(1))=\SO_{\Frob_{q^{n}}}(P^2(X, \Q_\ell)(1))
    \]
    and hence
    $
    I(\p(X_{\F_{q^m}})) = I(\p(X_{\F_{q^n}}))
    $
    for every $n$ divisible by $m$.
    \item $\tau_\ell$ maps $I^{\KS}(\Q)$ into $I(\p(X_{\F_{q^m}}))(\Q)$ under the identification $(\ref{equation:neutral component realization})$.
    Moreover, it descends to a surjection
    \[
    \tau \colon I^{\KS} \to I(\p(X_{\F_{q^m}}))
    \]
    over $\Q$. The kernel of $\tau$ is the subgroup of scalars $\G_m \subset I^{\KS}$, or in other words, $\tau$ induces an isomorphism $I^{\KS}/\G_m \cong I(\p(X_{\F_{q^m}}))$.
    \item The map $\tau(\Q) \colon I^{\KS}(\Q) \twoheadrightarrow I(\p(X_{\F_{q^m}}))(\Q)$ is surjective.
\end{enumerate}
\end{thm}

\begin{proof}
(1) Let $P_\ell:=P^2(X, \Q_\ell)(1)$.
As we already mentioned in Remark \ref{Remark:GSpin and SO},
we can embed $I^{\KS}_{\Q_\ell}$ into $\GSpin(P_\ell)$.
More precisely, by \cite[Corollary 9.9]{Ito-Ito-Koshikawa},
we can identify
$I^{\KS}_{\Q_\ell}$ with
the centralizer of a lift of $\Frob_{q^{n}}$ in $\GSpin(P_\ell)$ for all $n$ sufficiently divisible.
This fact implies $(1)$.

(2) It follows from the positivity of the Rosati involution that $(I^{\KS}/\G_m)(\R)$ is compact; see also the proof of \cite[Corollary 2.1.7]{KisinModp}.
Hence every element $\eta \in I^{\KS}(\Q)$ is semisimple and it is contained in a maximal torus of $I^{\KS}$ over $\Q$.
Thus, by \cite[Theorem 9.7]{Ito-Ito-Koshikawa} and the results of Mukai and Buskin \cite[Theorem 1.1]{Buskin} on the Hodge conjecture for rational Hodge isometries of K3 surfaces over $\C$,
there exist a finite extension $K$ of $W(\overline{\F}_q)[1/p]$ and a projective K3 surface $\mathcal{X}$ over $\mathcal{O}_K$ with a relatively ample line bundle $\mathcal{L}$ satisfying the following properties:
\begin{itemize}
    \item $(\mathcal{X}, \mathcal{L})$ is a lifting of $(X_{\overline{\F}_q}, \mathscr{L}_{\overline{\F}_q})$.
    \item For any embedding $K \hookrightarrow \C$,
    there exists an element
    $\tilde{\eta} \in \End(\p(\mathcal{X}_\C))$
    whose $\ell$-adic realization coincides with $\tau_\ell(\eta)$ via the identification
    \[
    P^2(\mathcal{X}_\C, \Q_\ell)(1) \cong P^2(X, \Q_\ell)(1).
    \]
    Here $P^2(\mathcal{X}_\C, \Q_\ell)(1)$ is the orthogonal complement of $\mathcal{L}_{\C}$ in $H^2_\et(\mathcal{X}_\C, \Q_\ell(1))$.
\end{itemize}
By using an analogue of Lemma \ref{Lemma:algebraically closed fields extension} and $(\ref{equation:endomorphism ring relative case})$,
we may regard $\tilde{\eta}$ as an element of
$\End(\p(\mathcal{X}))$
after enlarging $K$.
Then, by construction, the specialization of $\tilde{\eta}$ is contained in
$I(\p(X_{\F_{q^m}}))$ (using (1)) and is equal to $\tau_\ell(\eta)$ under the identification $(\ref{equation:neutral component realization})$.
This proves the first statement.
For the second statement, we remark that $I^{\KS}(\Q)$ is Zariski dense in $I^{\KS}_{\overline{\Q}_\ell}$.
Therefore, the base change of $\tau_\ell$ to $\overline{\Q}_\ell$ is compatible with the action of $\Aut(\overline{\Q}_\ell/\Q)$ and $\tau_\ell$ descends to a unique surjection
$
I^{\KS} \to I(\p(X_{\F_{q^m}}))
$
over $\Q$.
The last assertion follows from the fact that the kernel of
$
\GSpin(P_\ell) \to \SO(P_\ell)
$
is $\G_{m, \Q_\ell}$.

(3) This follows from (2) and Hilbert's theorem 90.
\end{proof}

Using Theorem \ref{Theorem:compare with Kisin's algebraic group} and the positivity of the Rosati involution for Kuga--Satake abelian varieties,
we shall deduce that the involution $\iota_\R$ is a \textit{positive involution} of $\End(\mathfrak{t}(X))_\R$, i.e.\
we have
$\Tr_{\End(\mathfrak{t}(X))_\R/\R}(f\iota(f)) >0$
for every nonzero $f \in \End(\mathfrak{t}(X))_\R$.

To do this, we need a lemma.
Let $X$ be a K3 surface over $\F_q$ of finite height.
Recall that $\Q[\Fr_X]$ is a field and
$\End(\mathfrak{t}(X))$ is a central simple algebra over $\Q[\Fr_X]$ (Theorem \ref{Theorem:structure of endomorphism ring}).
Moreover Proposition \ref{Proposition:L-function of K3} implies that $\Q[\Fr_X]$ is a CM field, i.e.\ a purely imaginary quadratic extension of a totally real number field.
The restriction of $\iota$ to the center $\Q[\Fr_X]$ is the complex conjugation and  $\iota(\Fr_X)=\Fr^{-1}_X$.

\begin{lem}\label{Lemma:positive involution and compact unitary}
The involution $\iota_\R$ is a positive involution of $\End(\mathfrak{t}(X))_\R$ if and only if $I(\mathfrak{t}(X))(\R)$ is compact.
\end{lem}

\begin{proof}
We see that
$\End(\mathfrak{t}(X))_\R=\End(\mathfrak{t}(X)) \otimes_{\Q[\Fr_X]} (\Q[\Fr_X] \otimes_\Q \R)$
is the product of $d(X)$ copies of the matrix algebra $M_e(\C)$ and the involution $\iota_\R$ preserves each factor, where $e:=e(X)$.
Thus, it suffices to show that an involution $\iota'$ of $M_e(\C)$ whose restriction to the center $\C$ is the complex conjugation is a positive involution (viewing $M_e(\C)$ as an $\R$-algebra) if and only if the topological group
$
\{ \, A \in M_e(\C)^\times \, \vert \, A\iota'(A)=1 \, \}
$
is compact.
The proof of this is standard, but we include it for completeness.
There exists a hermitian matrix $U \in M_e(\C)^\times$ such that $\iota'(A)=U({^*}A)U^{-1}$
for every matrix $A \in M_e(\C)$
(see \cite[Proposition 2.18]{BookInvolution} for example).
Here ${^*}A$ is the conjugate transpose of $A$.
It follows from \cite[Proposition 11.24]{BookInvolution} that
$\iota'$ is a positive involution if and only if the hermitian form $\langle-, -\rangle_U$ on $\C^{\oplus e}$ associated with $U$ is definite.
On the other hand, since the above topological group is the unitary group associated with $\langle-, -\rangle_U$, it is compact if and only if $\langle-, -\rangle_U$ is definite.
This concludes the proof.
\end{proof}

\begin{cor}\label{Corollary:positive involution}
Let $X$ be a K3 surface over $\F_q$.
The involution $\iota_\R$ is a positive involution of $\End(\mathfrak{t}(X))_\R$.
\end{cor}

\begin{proof}
If $X$ is supersingular, then $\mathfrak{t}(X)=0$ and there is nothing to check.
We may assume that $X$ is of finite height.
By Lemma \ref{Lemma:positive involution and compact unitary}, it suffices to show that $I(\mathfrak{t}(X))(\R)$ is compact.
After enlarging $\F_q$, we may assume that we are in the situation of Theorem \ref{Theorem:compare with Kisin's algebraic group}.
As stated in the proof of Theorem \ref{Theorem:compare with Kisin's algebraic group},
the positivity of the Rosati involution implies that
\[
(I^{\KS}/\G_m)(\R) \cong I(\p(X_{\F_{q^m}}))(\R)
\]
is compact, which in turn implies that $I(\mathfrak{t}(X))(\R)$ is also compact.
\end{proof}

\subsection{CM liftings of K3 surfaces over finite fields of finite height}\label{Subsection:CM liftings of K3 surfaces over finite fields of finite height}

In this subsection, we rewrite the main results in \cite{Ito-Ito-Koshikawa} on CM liftings without involving the Kuga--Satake construction and Kisin's algebraic group $I^{\KS}$.

Let us first recall the definition of K3 surfaces with \textit{complex multiplication} (CM) in characteristic $0$.
For a K3 surface $Y$ over an algebraically closed field, we denote by $\rho(Y):=\dim_\Q \Pic(Y)$ its Picard number.
We say that a (projective) K3 surface $Y$ over $\C$ has CM if
$\End_{\mathrm{Hdg}}(T(Y)_\Q)$ is a CM field whose degree over $\Q$ is equal to
$
22-\rho(Y)=\dim_\Q T(Y)_\Q$; see the proof of Proposition \ref{Proposition:transcendental motive is simple in characteristic 0} for the notation used here.
The Hodge conjecture holds true for the square $Y \times Y$; see \cite[Corollary 1.3]{Buskin}.
The Betti realization therefore induces an isomorphism
\begin{equation}\label{equation;Betti realization}
\End(\mathfrak{t}(Y)) \overset{\sim}{\to} \End_{\mathrm{Hdg}}(T(Y)_\Q).
\end{equation}
More generally, we make the following definition.

\begin{defn}\label{Definition:CM K3}
We say that a K3 surface $Y$ over a field $K$ of characteristic $0$ has CM if
$22-\rho(Y_{\overline{K}})=\dim_\Q \End(\mathfrak{t}(Y))$
and $\End(\mathfrak{t}(Y))$ is a CM field.
If furthermore $\End(\mathfrak{t}(Y))$ is isomorphic to a given CM field $E$, we say that $Y$ has CM by $E$.
\end{defn}

\begin{rem}\label{Remark:CM K3}
Let $Y$ be a K3 surface over $K$.
\begin{enumerate}
    \item Assume that $K=\C$. It follows from the results of Zarhin in \cite{Zarhin83} that $Y$ has CM if and only if the Mumford--Tate group of $T(Y)_\Q$ is commutative.
    \item Assume that $K$ can be embedded into $\C$.
    In \cite{Ito-Ito-Koshikawa}, we say that $Y$ has CM if
    $Y_\C$ has CM for some (and hence any) embedding $K \hookrightarrow \C$. (See \cite[Remark 9.5]{Ito-Ito-Koshikawa}.)
    In this case, there exists a finite extension $L$ of $K$ such that $Y_L$ has CM in the sense of Definition \ref{Definition:CM K3} by Lemma \ref{Lemma:algebraically closed fields extension} and ($\ref{equation;Betti realization}$).
    \item Assume that $K$ is algebraically closed for simplicity.
    It follows from \cite[Theorem 4]{Pjateckii-Shapiro-Shafarevich} that if $Y$ has CM then it is defined over a number field.
\end{enumerate}
\end{rem}

\begin{rem}\label{Remark:totally real or CM}
Let $Y$ be a K3 surface over $K$.
The number field $\End(\mathfrak{t}(Y))$ is in general either totally real or a CM field.
Moreover, the involution $\iota$ of $\End(\mathfrak{t}(Y))$ obtained by $f \mapsto {^t}f$ (cf.\ Section \ref{Subsection:Motivic isometries}) is a positive involution.
Indeed, we may assume that $K=\C$ as in Proposition \ref{Proposition:transcendental motive is simple in characteristic 0}.
Then the results follow from the Hodge-Riemann relation; see also the proof of \cite[Theorem 1.5.1]{Zarhin83}.
\end{rem}

We can now state the announced reformulation of our previous result on CM liftings.

\begin{thm}\label{Theorem:CM lifting revisited}
Let $X$ be a K3 surface over $\F_q$ of finite height.
Let $E \subset \End(\mathfrak{t}(X))$ be a maximal subfield that is stable under $\iota$.
Then $E$ is a CM field and there exist a finite extension $K$ of $W(\F_q)[1/p]$ with residue field $k$ and a projective K3 surface $\mathcal{X}$ over $\mathcal{O}_K$ satisfying the following properties:
\begin{enumerate}
    \item $\mathcal{X}$ is a lifting of $X_k$. 
    \item The specialization map
    $
    \Pic(\mathcal{X}_{\overline{K}})_\Q \to \Pic(\mathcal{X}_{\overline{\F}_q})_\Q
    $
    is an isomorphism, or equivalently
    $\mathrm{sp}(\mathfrak{t}(\mathcal{X}))=\mathfrak{t}({\mathcal{X}_{k}})$.
    \item The homomorphism
    $\End(\mathfrak{t}(\mathcal{X})) \to \End(\mathfrak{t}({\mathcal{X}_{k}}))$ (see $(\ref{equation:specialization of motivic homomorphisms})$) induces
    \[
    \End(\mathfrak{t}(\mathcal{X})) \overset{\sim}{\to} E \subset \End(\mathfrak{t}({\mathcal{X}_{k}})).
    \]
\end{enumerate}
In particular, the generic fiber $\mathcal{X}_K$ has CM by $E$.
\end{thm}

\begin{proof}
We first make a few remarks on a maximal subfield $E$ of $\End(\mathfrak{t}(X))$ that is stable under $\iota$.
Notice that $E$ must contain the center $\Q[\Fr_X]$.
Since $\iota$ is a positive involution by Corollary \ref{Corollary:positive involution}, it follows that $E$ is a CM field and the restriction of $\iota$ to $E$ is the complex conjugation.
Moreover, 
$\dim_\Q E=22-\rho(X_{\overline{\F}_q})$
by Corollary \ref{Corollary:rank of motivic endomorphism ring}.
In particular,
the image of $E$ in $\End(\mathfrak{t}(X_{\F_{q^m}}))$ (denoted by the same symbol) is still a maximal subfield.

In order to prove the theorem, we may replace $\F_q$ by its finite extension.
Thus we may assume that $X$ admits a primitive ample line bundle $\mathscr{L}$
and we can attach Kisin's algebraic group $I^{\KS}$ to $(X, \mathscr{L})$.
Moreover, we may assume that the assertions of Theorem $\ref{Theorem:compare with Kisin's algebraic group}$ hold for $m=1$.

There exists an element $\beta \in E$ such that $\beta \iota(\beta)=1$ and $E=\Q(\beta)$; see for example the proof of \cite[Theorem 3.7]{Huybrechts}.
By Theorem \ref{Theorem:compare with Kisin's algebraic group} (3), we can find an element
$\eta \in I^{\KS}(\Q)$
with
\[
\tau(\eta)=(1, \beta) \in \SO_{\Frob_q}((\mathscr{L})^\perp)(\Q) \times I(\mathfrak{t}(X))(\Q)=I(\p(X))(\Q).
\]
Then, the proof of Theorem \ref{Theorem:compare with Kisin's algebraic group} (2) shows that
there exist a finite extension $K$ of $W(\overline{\F}_q)[1/p]$, a lifting $(\mathcal{X}, \mathcal{L})$ of $(X_{\overline{\F}_q}, \mathscr{L}_{\overline{\F}_q})$, and an element
$\tilde{\eta} \in \End(\p(\mathcal{X}))$ that specializes to
$(1, \beta)$ in $\End(\p(X_{\overline{\F}_q}))$.
As in the proof of \cite[Theorem 9.7]{Ito-Ito-Koshikawa},
we can show that $\mathcal{X}_K$ has CM by checking the commutativity of the Mumford--Tate group of $T(\mathcal{X}_{\C})_\Q$ for any embedding $K \hookrightarrow \C$.
This implies that
$
\Pic(\mathcal{X}_{\overline{K}})_\Q \overset{\sim}{\to} \Pic(\mathcal{X}_{\overline{\F}_q})_\Q;
$
see \cite[Theorem 1.1]{Ito-K20}.
Here we give a different proof of these results.

Let $\gamma \in \End(\mathfrak{t}(\mathcal{X}))$ denote the restriction of $\tilde{\eta}$ to $\mathfrak{t}(\mathcal{X})$.
We first prove that
$
\Pic(\mathcal{X}_{\overline{K}})_\Q \overset{\sim}{\to} \Pic(\mathcal{X}_{\overline{\F}_q})_\Q.
$
It suffices to show the surjectivity of the map.
If it is not surjective, 
then the Betti realization $\gamma_B-1$ of $\gamma-1$ is not injective since $(\gamma_B-1)\otimes_\Q \Q_\ell$ is zero on the orthogonal complement of
$
\Pic(\mathcal{X}_{\overline{K}})_{\Q_\ell} \hookrightarrow \Pic(\mathcal{X}_{\overline{\F}_q})_{\Q_\ell}.
$
This implies that $\gamma_B-1=0$ since $T(\mathcal{X}_\C)$ is an irreducible $\Q$-Hodge structure.
This contradicts $\beta \neq 1$.

We have proved that $\mathrm{sp}(\mathfrak{t}(\mathcal{X}))=\mathfrak{t}({\mathcal{X}_{\overline{\F}_q}})$.
By construction, we see that $\gamma$ specializes to $\beta$.
Since $\End(\mathfrak{t}(\mathcal{X}))$ is a number field and
$E=\Q(\beta) \subset \End(\mathfrak{t}({\mathcal{X}_{\overline{\F}_q}}))$
is a maximal subfield, we obtain
$
\End(\mathfrak{t}(\mathcal{X})) \overset{\sim}{\to} E \subset \End(\mathfrak{t}({\mathcal{X}_{\overline{\F}_q}})).
$
This, together with
$
\Pic(\mathcal{X}_{\overline{K}})_\Q \overset{\sim}{\to} \Pic(\mathcal{X}_{\overline{\F}_q})_\Q
$, implies that $\mathcal{X}_K$ has CM by $E$.

Finally, as $\mathcal{X}_K$ has CM,
it follows from Remark \ref{Remark:CM K3} (3) and Lemma \ref{Lemma:algebraically closed fields extension} that (after enlarging $K$ if necessary)
the polarized K3 surface $(\mathcal{X}, \mathcal{L})$ has a model over $\mathcal{O}_L$ for some finite extension $L$ of $W(\F_q)[1/p]$ (contained in $K$) and 
the conditions of the theorem are satisfied for this model; see also \cite[Corollary 9.10]{Ito-Ito-Koshikawa}.
\end{proof}

\section{The Hodge standard conjecture for the square of a K3 surface}\label{Section:The Hodge standard conjecture for the square of a K3 surface}

In this section, we prove Theorem \ref{Theorem:Hodge standard for squares, intro}.
Let $X$ be a K3 surface over a field $k$.
The Hodge standard conjecture in characteristic $0$ holds true.
So we may assume that $k$ is of characteristic $p>0$.
By spreading out and specialization arguments, we may further assume that $k=\F_q$.

It is already known that the pairing $\langle -, - \rangle_n$ defined in Conjecture \ref{Conjecture:Hodge standard} is positive definite for $n=0, 1$; the case $n=1$ can be reduced to the Hodge index theorem for surfaces.
We shall show that $\langle -, - \rangle_2$ is positive definite.
Corollary \ref{Corollary:Conjecture D for squares over finite fields} allows us to work with algebraic cycles modulo numerical equivalence.
Since $X^2$ is four-dimensional,
it suffices to prove (after enlarging $\F_q$) that the signature of the intersection pairing
\[
\langle -, - \rangle \colon A^2(X^2) \times A^2(X^2) \to \Q, \quad (\alpha, \beta) \mapsto \langle \alpha, \beta \rangle
\]
is $(\rho_2-\rho_1 +1, \rho_1 -1)$
by \cite[Proposition 3.15]{Ancona}, where $\rho_n:= \dim_\Q A^n(X^2)$.

By \cite[Proposition 1.3.6]{Kleiman68}, the above intersection pairing can be computed as
\begin{equation}\label{equation:trace formula}
    \langle \alpha, \beta \rangle = \Tr_0({^t}\beta \circ \alpha) + \Tr_2({^t}\beta \circ \alpha) + \Tr_4({^t}\beta \circ \alpha), 
\end{equation}
where $\Tr_i({^t}\beta \circ \alpha)$ denotes the trace of the map $H^i(X, \Q_\ell) \to H^i(X, \Q_\ell)$ induced by ${^t}\beta \circ \alpha$.
Let $M$ denote the two-dimensional subspace of $A^2(X^2)$ generated by $\pi_0$ and $\pi_4$.
We obtain the following orthogonal decomposition
\[
A^2(X^2)= M \oplus \pi_{\alg} A^2(X^2) \pi_{\alg} \oplus \pi_{\tr} A^2(X^2) \pi_{\tr}.
\]
The signature of $\langle -, - \rangle$ on $M$ is $(1, 1)$.
After enlarging $\F_q$,
we have
\[
\pi_{\alg} A^2(X^2) \pi_{\alg} \cong \Pic(X_{\overline{\F}_q})_\Q \otimes_\Q \Pic(X_{\overline{\F}_q})_\Q
\]
and we can check that
the signature of $\langle -, - \rangle$
on $\pi_{\alg} A^2(X^2) \pi_{\alg}$ is $(\rho^2-\rho_1 +2, \rho_1-2)$
by using the Hodge index theorem, where
$\rho:=\dim_\Q \Pic(X_{\overline{\F}_q})_\Q=\rho_1/2$
is the Picard number of $X_{\overline{\F}_q}$.
Consequently, it suffices to show that
$\langle -, - \rangle$ is positive definite
on $\pi_{\tr} A^2(X^2) \pi_{\tr}= \End(\mathfrak{t}(X))$.
(In order to reach this conclusion, one can also use Lemma \ref{Lemma:reduction to transcendental} in the next section instead of \cite[Proposition 3.15]{Ancona}.)

Let $f \in \End(\mathfrak{t}(X))$ be a nonzero element.
We need to show that $\langle f, f \rangle >0$.
By $(\ref{equation:trace formula})$, we have
$\langle f, f \rangle = \Tr_2(f \iota(f))$.
We claim that
\begin{equation}\label{equation:trace equality}
    \Tr_{\End(\mathfrak{t}(X))/\Q}(f \iota(f)) = e \Tr_2(f \iota(f)),
\end{equation}
where $e:=e(X)$ (see Definition \ref{Definition:d, e, h}).
Let $T(X)_\ell \otimes_{\Q_\ell} \overline{\Q}_\ell=\oplus_i T_i$
be the decomposition into the eigenspaces $T_i$ of $\Frob_q$.
Then we obtain
\[
\End(\mathfrak{t}(X)) \otimes_\Q \overline{\Q}_\ell \overset{(\ref{equation:endomorphism ring realization})}{=} \End_{\Frob_q}(T(X)_\ell \otimes_{\Q_\ell} \overline{\Q}_\ell) = \prod_i \End_{\overline{\Q}_\ell}(T_i).
\]
Since each $T_i$ is of dimension $e$, the desired equality $(\ref{equation:trace equality})$ follows from the fact that $e$ times the trace of a matrix $A \in M_e(\overline{\Q}_\ell)$ is equal to
$\Tr_{M_e(\overline{\Q}_\ell)/\overline{\Q}_\ell}(A)$.
Finally, the positivity $\langle f, f \rangle >0$ follows from Corollary \ref{Corollary:positive involution} and $(\ref{equation:trace equality})$.

\section{Self-products of neat K3 surfaces}\label{Section:Self-products of neat K3 surfaces}

In this section, we study a certain class of K3 surfaces over finite fields.
Let $X$ be a K3 surface over $\F_q$ and consider a condition that, for any $n$,
the Tate cycles on the $n$-fold self-product $X^n:=X \times \cdots \times X$ are spanned by pull-backs of Tate cycles on $X^2$.
Such a K3 surface is called a \textit{neat} K3 surface in this paper (see Definition \ref{Definition:neat K3} below).
For example, $X$ is neat if it is ordinary \cite{Zarhin93} or supersingular.
We show the Tate conjecture and the Hodge standard conjecture for arbitrary powers $X^n$ of a neat K3 surface $X$ by reduction to the case of $X^2$.
As another example, we prove that a K3 surface over $\F_q$ with (geometric) Picard number $\rho \geq 18$ is neat, and consequently, the Hodge standard conjecture holds true for arbitrary powers of a K3 surface over an algebraically closed field with Picard number $\rho \geq 17$.

\subsection{Neat K3 surfaces}\label{Subsection:Neat K3 surfaces}

Let $X$ be a K3 surface over $\F_q$.
We assume for the moment that $X$ is of finite height.
We write $d:=d(X)$ (see Definition \ref{Definition:d, e, h}).
Let $\alpha_1, \dotsc, \alpha_{2d}$ denote the \textit{distinct} eigenvalues of $\Frob_q$ acting on $T(X)_\ell$.
In other words, they are the roots of the irreducible polynomial $Q(T)$ in Proposition \ref{Proposition:L-function of K3}.
We may and do assume that
$
\alpha_{i+d}=\alpha^{-1}_i
$
for every $1 \leq i \leq d$.
Let $\Gamma(X)$ denote the multiplicative subgroup of $\overline{\Q}^\times$ generated by $\alpha_1, \dotsc, \alpha_{d}$, so it is a finitely generated abelian group with
\[
1 \leq \rank \Gamma(X) \leq d(X).
\]
If $X$ is supersingular, we let $\Gamma(X)$ denote the trivial group and $d(X):=0$.

Following the work of Zarhin \cite{Zarhin94, Zarhin15} on abelian varieties over finite fields, we introduce the following definition.
Let us remark that
$\rank \Gamma(X_{\F_{q^m}})=\rank \Gamma(X)$ and $d(X_{\F_{q^m}}) \leq d(X)$ for a finite extension $\F_{q^m}$ of $\F_q$.

\begin{defn}\label{Definition:neat K3}
We say that a K3 surface $X$ over $\F_q$ is \textit{neat} if
$\rank \Gamma(X_{\F_{q^m}})=d(X_{\F_{q^m}})$ for some finite extension $\F_{q^m}$ of $\F_q$.
\end{defn}

\begin{rem}\label{Remark:neat K3}
\begin{enumerate}
\item Assume that $X$ is of finite height. With the above notation, the following are equivalent.
\begin{enumerate}
    \item The equality $\rank \Gamma(X)=d(X)$ holds.
    \item The natural surjection
    $\alpha^\Z_1 \cdots \alpha^\Z_d \to \Gamma(X)$ is an isomorphism, where $\alpha^\Z_1 \cdots \alpha^\Z_d$ denotes the free abelian group with the basis $\alpha_1, \dotsc, \alpha_d$.
    \item For any function $f \colon \{ \alpha_1, \dotsc, \alpha_{2d} \} \to \Z$ that satisfies $\prod_{1 \leq i \leq 2d} \alpha^{f(\alpha_i)}_i=1$, the equality $f(\alpha_i)=f(\alpha_{i+d})$ holds for every $1 \leq i \leq d$.
\end{enumerate}
\item Let $\F_{q^m}$ be a finite extension of $\F_q$.
A K3 surface $X$ over $\F_q$ is neat if and only if $X_{\F_{q^m}}$ is neat.
\end{enumerate}
\end{rem}

Zarhin proved that every ordinary K3 surface $X$ (i.e.\ a K3 surface of height $1$) over $\F_q$ is neat in \cite{Zarhin93}.
He later used the neatness to prove the Tate conjecture for arbitrary powers $X^n$ of an ordinary $X$ by reducing it to the case of $X^2$ \cite{Zarhin96}.
In fact, the same argument applies to any neat K3 surface thanks to the Tate conjecture for $X^2$ (Theorem \ref{Theorem:Tate conjecture for square}), as explained in \cite[Theorem 5.4 (a)]{Milne19}.
(The condition that $\Fr_X$ is regular in the sense of \cite{Milne19} is equivalent to the requirement that $\rank \Gamma(X)=d(X)$; see \cite[Aside 2.7]{Milne19}.)
Let us briefly recall the argument.

\begin{prop}[\cite{Milne19}]\label{Proposition:Tate conjecture for neat K3}
Let $X$ be a neat K3 surface over $\F_q$ and $n$ a positive integer.
The Tate conjecture for $X^n$ holds true. Moreover, numerical equivalence coincides with $\ell$-adic homological equivalence for algebraic cycles on $X^n$.
The same statements hold for the crystalline cohomology.
\end{prop}

\begin{proof}
Since we already know that the Tate conjecture for both $X$ and $X^2$ holds true, we may assume that $n \geq 3$.
After enlarging $\F_q$, we may further assume that $\Frob_q$ acts trivially on $\Pic(X_{\overline{\F}_q})$ and $\rank \Gamma(X)=d(X)$.
Then we can show as in \cite{Zarhin96} that every $\Frob_q$-invariant element in $H^{2i}(X^n, \Q_\ell)(i)$ is a linear combination of cup products of pull-backs of $\Frob_q$-invariant elements in $H^{2j}(X^2, \Q_\ell)(j)$ ($0 \leq j \leq 4$) via projection maps $X^n \to X^2$.
In particular, the Tate conjecture for $X^n$ also holds true.
This, together with the semisimplicity of Frobenius, implies the remaining assertions by the same argument as in Corollary \ref{Corollary:Conjecture D for squares over finite fields}.
\end{proof}

\begin{rem}
Let $X$ be a (not necessarily neat) K3 surface over $\F_q$. 
According to \cite[Theorem 5.4 (b)]{Milne19}, there exist infinitely many $\ell$ such that numerical equivalence coincides with $\ell$-adic homological equivalence for algebraic cycles on $X^n, n \geq 1$. This result is an analogue of the result of Clozel \cite{Clozel} for abelian varieties over finite fields, and relies on Theorem \ref{Theorem:Tate conjecture for square}. 
\end{rem}

\subsection{The Hodge standard conjecture for self-products of neat K3 surfaces}\label{Subsection:The Hodge standard conjecture for self-products of neat K3 surfaces}

In this subsection, we first explain that, in order to prove the Hodge standard conjecture for powers of a K3 surface over a finite field, it suffices to establish the positivity of certain pairings on the tensor powers of the transcendental motive.
Then, we deduce this positivity for powers of a neat K3 surface from Theorem \ref{Theorem:Hodge standard for squares, intro}.

Let $k$ be a finite field $\F_q$ or $\overline{\F}_q$ and let $X$ be a K3 surface over $k$.
We fix a prime number $\ell$ invertible on $X$.
For the moment, we work with $\ell$-adic homological motives over $k$.
The argument in Section \ref{Subsection:The transcendental motive} also applies to $\ell$-adic homological motives, so we can define the
$\ell$-adic homological transcendental motive $\mathfrak{t}(X) \in \mathscr{M}_{\hom}(k)$, denoted by the same notation.
Let $\mathfrak{t}(X)^{\otimes i}$ denote 
the $i$-th tensor power of $\mathfrak{t}(X)$.
We consider the $\Q$-vector space
$
\Hom_{\mathscr{M}_{\hom}(k)}(\mathbbm{1}, \mathfrak{t}(X)^{\otimes i})
$
and equip it with the pairing $\langle -, - \rangle^{\otimes i}$ induced from the natural (cup product) pairing
$
\langle -, - \rangle \colon \mathfrak{t}(X) \otimes \mathfrak{t}(X) \to \mathbbm{1}.
$

\begin{lem}\label{Lemma:reduction to transcendental}
Let $X$ be a K3 surface over $k$ and $n$ a positive integer.
We fix a prime number $\ell$ invertible on $X$.
With the above notation, the following are equivalent.
\begin{enumerate}
    \item Let $\mathscr{L}$ be an ample line bundle on $X^n$. Conjecture \ref{Conjecture:Hodge standard} holds true for $X^n$, $\mathscr{L}$, and $\ell$.
    \item For every $i \leq n$,
    the pairing $\langle -, - \rangle^{\otimes i}$ on
    $
    \Hom_{\mathscr{M}_{\hom}(k)}(\mathbbm{1}, \mathfrak{t}(X)^{\otimes i})
    $
    is positive definite if $i$ is even and is negative definite if $i$ is odd.
\end{enumerate}
In particular, the validity of Conjecture \ref{Conjecture:Hodge standard} for $X^n$ does not depend on the choice of the ample line bundle $\mathscr{L}$.
\end{lem}

\begin{proof}
We may assume that $k=\overline{\F}_q$.
Since the Picard scheme of $X$ is \'etale over $k$,
it follows that every ample line bundle $\mathscr{L}$ on $X^n$ is of the form $\mathscr{L}=\sum_{1 \leq m \leq n} p^*_m \mathscr{L}_m$, where $p_m \colon X^n \to X$ is the $m$-th projection and $\mathscr{L}_m$ is an ample line bundle on $X$.

As the Lefschetz standard conjecture holds true for $X^n$ (see Remark \ref{Remark:Conjecture B}),
we have a decomposition of the $\ell$-adic homological motive $\mathfrak{h}^i(X^n)$ for any $i$
\[
\mathfrak{h}^i(X^n) = \bigoplus_{j \geq \max(i-2n, 0)} \mathfrak{p}^{i-2j}(X^n)(-j)
\]
whose $\ell$-adic realization gives the Lefschetz decomposition of
$H^i(X^n, \Q_\ell)$ with respect to $\mathscr{L}$.
Let
\[
*_{\mathrm{H}, \mathscr{L}} \colon \mathfrak{h}^i(X^n) \to \mathfrak{h}^{4n-i}(X^n)(2n-i)
\]
denote the involution given by multiplication by $(-1)^{(i-2j)(i-2j+1)/2}\frac{j!}{(2n-i+j)!}$ on each summand $\mathfrak{p}^{i-2j}(X^n)(-j)$.
The $\ell$-adic realization of $*_{\mathrm{H}, \mathscr{L}}$ coincides with the involution $*_{\mathrm{H}}$ (with respect to $\mathscr{L}$) defined in \cite[Section 1]{Andre96}.
We consider the following composition
\[
q^i_{\mathscr{L}} \colon \mathfrak{h}^{2i}(X^n)(i) \otimes \mathfrak{h}^{2i}(X^n)(i) \overset{\id \otimes *_{\mathrm{H}, \mathscr{L}}}{\longrightarrow} \mathfrak{h}^{2i}(X^n)(i) \otimes \mathfrak{h}^{4n-2i}(X^n)(2n-i) \to \mathbbm{1},
\]
where $\mathfrak{h}^{2i}(X^n)(i) \otimes \mathfrak{h}^{4n-2i}(X^n)(2n-i) \to \mathbbm{1}$ is the natural pairing.
This induces a quadratic form on
$
\Hom_{\mathscr{M}_{\hom}(k)}(\mathbbm{1}, \mathfrak{h}^{2i}(X^n)(i)),
$
denoted by the same notation $q^i_{\mathscr{L}}$.

It is well known that
Conjecture \ref{Conjecture:Hodge standard} for $X^n$, $\mathscr{L}$, and $\ell$ holds true if and only if $q^i_{\mathscr{L}}$ is positive definite on
$\Hom_{\mathscr{M}_{\hom}(k)}(\mathbbm{1}, \mathfrak{h}^{2i}(X^n)(i))$
for any $i$; see also \cite[Conjecture 5.3.2.1]{Andre04}.
Indeed, we can check this by using the fact that the Lefschetz decomposition
\[
\Hom_{\mathscr{M}_{\hom}(k)}(\mathbbm{1}, \mathfrak{h}^{2i}(X^n)(i))=A^i_\ell(X^n) = \bigoplus_{j \geq \max(2i-2n, 0)} A^{i-j, \prim}_\ell(X^n)
\]
is orthogonal with respect to $q^i_{\mathscr{L}}$ and the bilinear form associated with $q^i_{\mathscr{L}}$ is equal to the pairing 
$\langle -, - \rangle_{i-j}$ on $A^{i-j, \prim}_\ell(X^n)$ up to a positive multiplicative constant.

We have a decomposition
\[
\mathfrak{h}^{2i}(X^n)(i) = \bigoplus_{(i_1, \dotsc, i_n)} \bigotimes_{1 \leq m \leq n} \mathfrak{h}^{2i_m}(X)(i_m),
\]
where $(i_1, \dotsc, i_n) \in \Z^n_{\geq 0}$
runs over the set of $n$-tuples of nonnegative integers with
$\sum_{1 \leq m \leq n} i_m = i$.
Using a formula for $*_{\mathrm{H}}$ for the product of two varieties given in \cite[Lemme 1.3.2]{Andre96}, we see that
the restriction of $q^i_{\mathscr{L}}$ to
\[
(\otimes_{1 \leq m \leq n} \mathfrak{h}^{2i_m}(X)(i_m)) \otimes (\otimes_{1 \leq m \leq n} \mathfrak{h}^{2i'_m}(X)(i'_m))
\]
is
$\otimes_{1 \leq m \leq n} q^{i_m}_{\mathscr{L}_m}$ if
$(i_1, \dotsc, i_n)=(i'_1, \dotsc, i'_n)$ and is zero otherwise.

By the results in the previous two paragraphs,
we need only show that the quadratic form
on
\[
V(i_1, \dotsc, i_n):=\Hom_{\mathscr{M}_{\hom}(k)}(\mathbbm{1}, \otimes_{1 \leq m \leq n} \mathfrak{h}^{2i_m}(X)(i_m))
\]
induced by $\otimes_{1 \leq m \leq n} q^{i_m}_{\mathscr{L}_m}$ is positive definite for every $(i_1, \dotsc, i_n) \in \Z^n_{\geq 0}$
if and only if the pairing $(-1)^i\langle -, - \rangle^{\otimes i}$ on
$
\Hom_{\mathscr{M}_{\hom}(k)}(\mathbbm{1}, \mathfrak{t}(X)^{\otimes i})
$
is positive definite for every $i \leq n$.
Notice that the involution $*_{\mathrm{H}, \mathscr{L}_m}$ acts by multiplication by $-1$ on $\mathfrak{t}(X)$.
The quadratic space
$
V(i_1, \dotsc, i_n)
$
is the direct sum of quadratic spaces of the form
\[
A^0_\ell(X)^{\otimes j_1} \otimes_\Q A^1_\ell(X)^{\otimes j_2} \otimes_\Q \Hom_{\mathscr{M}_{\hom}(k)}(\mathbbm{1}, \mathfrak{t}(X)^{\otimes j_3}) \otimes_\Q A^2_\ell(X)^{\otimes j_4}
\]
with $\sum j_m = n$.
Here we equip each $A^i_\ell(X)$ with the quadratic form induced from $q^i_{\mathscr{L}_m}$ for some $\mathscr{L}_m$, which in particular is positive definite (for $i=1$ this is the Hodge index theorem), and equip $\Hom_{\mathscr{M}_{\hom}(k)}(\mathbbm{1}, \mathfrak{t}(X)^{\otimes j_3})$ with the pairing $(-1)^{j_3}\langle -, - \rangle^{\otimes j_3}$.
This description of
$
V(i_1, \dotsc, i_n)
$
suffices to conclude our proof.
\end{proof}

\begin{thm}\label{Theorem:Hodge standard conjecture for neat K3 surface}
Let $X$ be a neat K3 surface over $\F_q$.
Conjecture \ref{Conjecture:Hodge standard} holds true for an arbitrary power $X^n$ of $X$ and every ample line bundle on $X^n$.
\end{thm}

\begin{proof}
By Lemma \ref{Lemma:reduction to transcendental}, it is enough to prove that, for every $n$, the pairing $(-1)^n\langle -, - \rangle^{\otimes n}$ on
$\Hom_{\mathscr{M}_{\hom}(\F_q)}(\mathbbm{1}, \mathfrak{t}(X)^{\otimes n})
$
is positive definite.
Since we have
\[
\Hom_{\mathscr{M}_{\hom}(\F_q)}(\mathbbm{1}, \mathfrak{t}(X)^{\otimes n})=
\Hom_{\mathscr{M}_{\num}(\F_q)}(\mathbbm{1}, \mathfrak{t}(X)^{\otimes n})
\]
by Proposition \ref{Proposition:Tate conjecture for neat K3}, we are allowed to work with numerical motives.
We may assume that $X$ is of finite height since otherwise $\mathfrak{t}(X)=0$.
After enlarging $\F_q$, we may further assume that $\rank \Gamma(X)=d(X)$.
Let $\alpha_1, \dotsc, \alpha_{2d} \in \overline{\Q}$ denote the distinct eigenvalues of $\Frob_q$ acting on $T(X)_\ell$ with
$
\alpha_{i+d}=\alpha^{-1}_i
$
for every $1 \leq i \leq d=d(X)$.

We first give an orthogonal decomposition of
$\Hom_{\mathscr{M}_{\num}(\F_q)}(\mathbbm{1}, \mathfrak{t}(X)^{\otimes n}) \otimes_\Q \R$ and then explain that the positivity of each summand can be deduced from that of
\[
\Hom_{\mathscr{M}_{\num}(\F_q)}(\mathbbm{1}, \mathfrak{t}(X)^{\otimes 2})\cong\End(\mathfrak{t}(X)),
\]
which is proved in Theorem \ref{Theorem:Hodge standard for squares, intro}.
Let
$\mathscr{M}_{\num}(\F_q, \C)$
denote the category of numerical motives over $\F_q$ with coefficients in $\C$ and the scalar extension of a numerical motive $M$ is denoted by $M_{\C}$.
The field $\Q[\Fr_X]$ is isomorphic to each $\Q(\alpha_i)$ so that we have the following isomorphism of $\C$-algebras
\[
\Q[\Fr_X] \otimes_\Q \C \cong \prod_{1 \leq i \leq 2d} \C, \quad \Fr_X \otimes 1 \mapsto (\alpha_1, \dotsc, \alpha_{2d}).
\]
This induces a natural decomposition of $\mathfrak{t}(X)_{\C}$ in $\mathscr{M}_{\num}(\F_q, \C)$
\[
\mathfrak{t}(X)_{\C} = \bigoplus_{1 \leq i \leq 2d} M({\alpha_i}).
\]
For an $n$-tuple $I=(\alpha_{i_1}, \dotsc, \alpha_{i_n})$ of the eigenvalues, we put
$
M_I:= \bigotimes_{1 \leq m \leq n} M(\alpha_{i_m})
$
and
$V_I:= \Hom_{\mathscr{M}_{\num}(\F_q, \C)}(\mathbbm{1}_{\C}, M_I)$.
We obtain a decomposition
$
\mathfrak{t}(X)^{\otimes n}_{\C} = \oplus_{I} M_I.
$
By the assumption that $\rank \Gamma(X)=d(X)$,
we see that
$\Hom_{\mathscr{M}_{\num}(\F_q)}(\mathbbm{1}, \mathfrak{t}(X)^{\otimes n}) \otimes_\Q \C=0$
if $n$ is odd.
Moreover, we can write
\begin{equation}\label{equation:orthogonal decomposition tensor power}
\Hom_{\mathscr{M}_{\num}(\F_q)}(\mathbbm{1}, \mathfrak{t}(X)^{\otimes 2n}) \otimes_\Q \C = \bigoplus_{J} V_J,
\end{equation}
where $J$ runs over the set of $2n$-tuples
$J=(\alpha_{i_1}, \dotsc, \alpha_{i_{2n}})$
of the eigenvalues such that, after permuting them, they satisfy $\alpha_{i_{2m}}=\alpha^{-1}_{i_{2m-1}}$ for every $1 \leq m \leq n$.
We let $\overline{J}:=(\alpha^{-1}_{i_1}, \dotsc, \alpha^{-1}_{i_{2n}})$.
Then, the action of the complex conjugation on the left hand side of $(\ref{equation:orthogonal decomposition tensor power})$ maps
$V_J$ to $V_{\overline{J}}$.
Therefore $V_J \oplus V_{\overline{J}}$ descends to an $\R$-vector subspace $V_{[J]}$ of
$\Hom_{\mathscr{M}_{\num}(\F_q)}(\mathbbm{1}, \mathfrak{t}(X)^{\otimes 2n}) \otimes_\Q \R$.
This induces an orthogonal decomposition
\[
\Hom_{\mathscr{M}_{\num}(\F_q)}(\mathbbm{1}, \mathfrak{t}(X)^{\otimes 2n}) \otimes_\Q \R = \bigoplus V_{[J]}.
\]

It suffices to prove that the pairing $\langle -, - \rangle^{\otimes 2n}$ is positive definite on $V_{[J]}$ for every $J=(\alpha_{i_1}, \dotsc, \alpha_{i_{2n}})$.
Since the pairing on $\Hom_{\mathscr{M}_{\num}(\F_q)}(\mathbbm{1}, \mathfrak{t}(X)^{\otimes 2n})$
is invariant under permutations of factors $\mathfrak{t}(X)$,
we may assume that $\alpha_{i_{2m}}=\alpha^{-1}_{i_{2m-1}}$ for every $1 \leq m \leq n$ after permuting $\alpha_{i_1}, \dotsc, \alpha_{i_{2n}}$.
By using Proposition \ref{Proposition:Tate conjecture for neat K3}, we obtain isomorphisms
\[
\bigotimes_{1 \leq m \leq n} V_{(\alpha_{i_{2m-1}}, \alpha_{i_{2m}})} \overset{\sim}{\to} V_J \quad \text{and} \quad \bigotimes_{1 \leq m \leq n} V_{(\alpha_{i_{2m}}, \alpha_{i_{2m-1}})} \overset{\sim}{\to} V_{\overline{J}}.
\]
It follows that there exists an embedding
$V_{[J]} \hookrightarrow \otimes_{1 \leq m \leq n} V_{[(\alpha_{i_{2m-1}}, \alpha_{i_{2m}})]}$
of quadratic spaces over $\R$.
By Theorem \ref{Theorem:Hodge standard for squares, intro}, we know that
the pairing $\langle -, - \rangle^{\otimes 2}$ is positive definite on
$V_{[(\alpha_{i_{2m-1}}, \alpha_{i_{2m}})]}$.
Therefore, the quadratic space
$\otimes_{1 \leq m \leq n} V_{[(\alpha_{i_{2m-1}}, \alpha_{i_{2m}})]}$,
and hence $V_{[J]}$, is positive definite.
The proof of Theorem \ref{Theorem:Hodge standard conjecture for neat K3 surface} is now complete.
\end{proof}

\subsection{Examples}\label{Subsection:Examples}

This subsection is devoted to discuss examples.
First, we give some examples of neat K3 surfaces:

\begin{prop}\label{Proposition:example of neat K3}
Let $X$ be a K3 surface over $\F_q$.
If one of the following conditions is satisfied after enlarging $\F_q$, then $X$ is neat.
\begin{enumerate}
    \item $X$ is of finite height and $h(X)=e(X)$ holds (e.g.\ if $X$ is ordinary).
    \item The inequality $d(X) \leq 2$ holds (e.g.\ if $\rho(X_{\overline{\F}_q}) \geq 18$).
\end{enumerate}
\end{prop}

\begin{proof}
(1) A corresponding statement is shown in \cite{Zarhin93}.
Indeed, after enlarging $\F_q$, we may assume that $h(X)=e(X)$.
It follows that the irreducible polynomial $Q(T)$ is a (rational multiple of) $p$-admissible polynomial of K3 type in the sense of \cite[Definition 2.3]{Zarhin93}.
Thus, $\rank \Gamma(X)=d(X)$ holds by \cite[Theorem 2.4.3]{Zarhin93}.

(2) After enlarging $\F_q$, we may assume that $d(X)\leq 2$ and $\Gamma(X)$ is torsion-free.
We shall prove $\rank \Gamma(X)=d(X)$ in this case.
If $X$ is supersingular, then $\rank \Gamma(X)=d(X)=0$ by definition.
We may assume that $X$ is of finite height, so $1 \leq \rank \Gamma(X) \leq d(X) \leq 2$.
If $\rank \Gamma(X)=2$, then $\rank \Gamma(X)=d(X)=2$ holds trivially.
If $\rank \Gamma(X)=1$, we can prove $d(X)=1$ by the same argument as in \cite[Lemma 2.10]{Zarhin94}.
We include here a slightly modified version for the convenience of the reader.
The torsion-free group $\Gamma(X)$ is isomorphic to $\Z$ by assumption.
Let $L$ denote the splitting field of $Q(T)$.
The Galois group $\Gal(L/\Q)$ naturally acts on
$\Gamma(X)\subset L$ and it induces a homomorphism
\[
\Gal(L/\Q) \to \{ 1, -1 \} \cong \Aut(\Gamma(X)).
\]
This map is injective as $\Gamma (X)$ generates $L$, so we must have $d(X)=1$.
\end{proof}

Recall that the height $h$ of a K3 surface $X$ over an algebraically closed field of characteristic $p>0$ (not necessarily $\overline{\F}_q$) is also defined as the height of the associated formal Brauer group.
We say that $X$ is ordinary if $h=1$, and that $X$ is supersingular if $h=\infty$.

\begin{cor}\label{Corollary:Hodge standard for Picard number 17}
Conjecture \ref{Conjecture:Hodge standard} holds true for an arbitrary power $X^n$ of a K3 surface $X$ over an algebraically closed field of characteristic $p>0$ and every ample line bundle on $X^n$ if $X$ satisfies one of the following conditions:
\begin{enumerate}
    \item $X$ is ordinary.
    \item $X$ is supersingular.
    \item $\rho(X) \geq 17$.
\end{enumerate}
\end{cor}
In fact, the condition (2) implies the condition (3) since we have $\rho(X)=22$ if $X$ is supersingular; see Remark \ref{Remark:supersingular case}.

\begin{proof}
In each case, the K3 surface $X$ and an ample line bundle $\mathscr{L}$ on $X^n$ specialize to a K3 surface $X_0$ over a finite field $\F_q$ and an ample line bundle $\mathscr{L}_0$ on $X^n_0$ such that $X_0$ satisfies one of the conditions in Proposition \ref{Proposition:example of neat K3}, and so that $X_0$ is neat.
In the case (1), this follows from the openness of the ordinary locus in a family of K3 surfaces; see \cite{Artin}.
In the case (3), this follows from the fact that the Picard number does not decrease under specialization and the geometric Picard number of a K3 surface defined over a finite field is even (hence $\rho((X_0)_{\overline{\F}_q}) \geq 17$ implies $\rho((X_0)_{\overline{\F}_q}) \geq 18$).
Conjecture \ref{Conjecture:Hodge standard} holds true for $X^n_0$ and $\mathscr{L}_0$ by Theorem \ref{Theorem:Hodge standard conjecture for neat K3 surface}, and it in turn implies that for $X^n$ and $\mathscr{L}$.
\end{proof}

\begin{ex}\label{Example:Kummer surface}
Let $A$ be an abelian surface over an algebraically closed field of characteristic $p \geq 3$ and $X$ the Kummer surface associated with $A$.
It is well known that $\rho(X)=\rho(A)+16 \geq 17$, where $\rho(A)$ is the Picard number of $A$, and so $\rho(X) \geq 18$ if $A$ is defined over a finite field $\F_q$.
Thus, the Hodge standard conjecture holds true for arbitrary powers $X^n$ by Corollary \ref{Corollary:Hodge standard for Picard number 17},
and $X$ is neat in the latter case by Proposition \ref{Proposition:example of neat K3}.

One may also deduce the neatness of $X$ from the fact that abelian surfaces over finite fields are neat, using a Galois-equivariant isomorphism
\begin{equation}\label{equation:Kummer surface transcendental}
T(A)_\ell \cong T(X)_\ell,
\end{equation}
where $T(A)_\ell \subset H^2_\et(A_{\overline{\F}_q}, \Q_\ell)(1)$ is the orthogonal complement of the N\'eron-Severi group of $A_{\overline{\F}_q}$.
See \cite[Definition 3.0]{Zarhin94} for the definition of a neat abelian variety and \cite[Theorem 3.5]{Zarhin94} for the proof of the neatness of abelian surfaces.
(Actually, the discussion given in Proposition \ref{Proposition:example of neat K3} (2) is almost the same as the proof of this fact.)
Furthermore, it is plausible that the Hodge standard conjecture for $X^2$ (and $X^n$) follows from the known case of $A^2$.  
\end{ex}

Next, we shall give an example of a nonneat K3 surface.
For this, we use a weak analogue of the Honda--Tate theory studied in \cite{Taelman, Ito-K19}, which we recall briefly below.

For a polynomial $L(T)= \prod_i (1-\alpha_i T) \in 1 + T\Q[T]$ and a positive integer $N$, let
$L(T)^{(N)}$ denote the polynomial
$
\prod_i (1-\alpha^N_i T) \in 1 + T\Q[T].
$
Taelman \cite{Taelman} conjectured that, if the degree of $L(T)$ is less than or equal to $20$ and $L(T)$ satisfies the conditions of Proposition \ref{Proposition:L-function of K3} for a power $q$ of $p$, then there exist a positive integer $N$ and a K3 surface $X$ over $\F_{q^N}$ such that
\[
L_{\mathrm{trc}}(X, T)=\det(1-\Frob_q^N T \vert T(X)_{\ell})=L(T)^{(N)}.
\]
Moreover, he proved that this conjecture holds true if any K3 surface with CM admits a certain semistable model.
By refining the argument of Taelman, the first author proved the conjecture under mild assumptions on $p$, which are satisfied if $p \geq 7$, or $p = 5$ and the degree of $L(T)$ is strictly less than $20$; see \cite[Theorem 1.3]{Ito-K19}.

In \cite[Section 4]{Zarhin94} and \cite[Section 7]{Zarhin15}, a nonneat simple abelian threefold is constructed via the Honda--Tate theory. 
With the above results, we can construct a nonneat K3 surface by the same argument as in \cite[Theorem 7.1]{Zarhin15} as follows.

\begin{ex}\label{Example:non neat}
Assume that $p \geq 3$.
As in \cite[Section 7]{Zarhin15},
we can find a Weil $q$-number $\alpha$ with the following properties:
\begin{itemize}
    \item $q$ is of the form $p^{2a}$ for a positive integer $a$.
    \item $\Q(\alpha)$ is a CM field of degree $6$, and $\Q(\alpha)=\Q(\alpha^n)$ for any positive integer $n$.
    \item Let $\{ \alpha_i \}_{1 \leq i \leq 6}$ denote the conjugates of $\alpha$ over $\Q$. Then $\nu_q(\alpha_i) \in \{ 0, 1/2, 1 \}$ for each $\alpha_i$ and the number of $\alpha_i$ with $\nu_q(\alpha_i)=1/2$ is $2$.
    \item After permuting $\{ \alpha_i \}_{1 \leq i \leq 6}$, the equality  $\alpha^2_1\alpha^2_2\alpha^2_3=q^3$ holds. 
    \item $\Q(\alpha)$ has exactly three distinct primes above $p$.
\end{itemize}
(Here we fix an embedding $\overline{\Q} \subset \overline{\Q}_p$ and
the $p$-adic valuation $\nu_q \colon \overline{\Q}_p \to \Q \cup \{ \infty \}$ is normalized by $\nu_q(q)=1$ as in Section \ref{Subsection:The height of a K3 surface}.)
So, the polynomial
\[
L(T):=\prod_{1 \leq i \leq 6} (1- \frac{\alpha_i}{\sqrt{q}}T)
\]
satisfies the conditions of Proposition \ref{Proposition:L-function of K3} for the power $q$ of $p$ (with $d=3, e=1, h=2$).
If $p \geq 5$, then there exist a positive integer $N$ and a K3 surface $X$ over $\F_{q^N}$ such that
$
L_{\mathrm{trc}}(X, T)=L(T)^{(N)}
$
by \cite[Theorem 1.3]{Ito-K19}.
One can easily check that
$\rank \Gamma(X)=2$ and
$d(X_{\F_{q^{m}}})=3$ for any finite extension $\F_{q^{m}}$ of $\F_{q^N}$,
and hence the resulting K3 surface $X$ is \textit{not} neat.
\end{ex}

We also include the following example showing that the integers $d(X)$ and $e(X)$ may change under base field extensions.

\begin{ex}\label{Example:d and e change}
Assume that $p \geq 3$.
We consider a supersingular elliptic curve $E_1$ over a finite field $\F_q$ of characteristic $p$
such that the eigenvalues $\alpha, q/\alpha$ of $\Frob_q$ acting on
$H^1_\et((E_1)_{\overline{\F}_q}, \Q_\ell)$ are distinct.
(For example, let $q$ be of the form $p^{2a+1}$ and $E_1$ the elliptic curve corresponding to the Weil $q$-number $\sqrt{-q}$.)
Let $E_2$ be an ordinary elliptic curve over $\F_q$
and $X$ a Kummer surface associated with $E_1 \times E_2$.
Then, using $(\ref{equation:Kummer surface transcendental})$,
we can see that
$L_{\mathrm{trc}}(X, T)$
is irreducible over $\Q$ and of degree $4$, namely, the equalities $d(X)=2$ and $e(X)=1$ hold.
On the other hand, $d(X_{\F_{q^m}}) = 1$ and $e(X_{\F_{q^m}})=2$ for a finite extension $\F_{q^m}$ of $\F_q$.
(The height $h(X)=h(X_{\F_{q^m}})$ of $X$ is $2$.)
\end{ex}

\begin{rem}
Let $p\geq 5$ be a prime number and $e$ a positive integer satisfying $1 \leq e \leq 10$.
We can produce a K3 surface $X$ over a finite field of characteristic $p$ with $e(X)=e$, using \cite[Theorem 1.3]{Ito-K19}.
See also \cite[Section 6]{Ito-K19}.
\end{rem}

\subsection*{Acknowledgements}
The authors would like to thank Ziquan Yang and James S.\ Milne
for comments on a preliminary version of this paper.

The work of the third author was supported by JSPS KAKENHI Grant Number 20K14284.
The work of the second and the third author was supported by JSPS KAKENHI Grant Number 21H00973.


\bibliographystyle{abbrvsort}
\bibliography{bibliography.bib}
\end{document}